\documentclass[final]{siamltex}
\usepackage{amssymb,amsmath}

\usepackage{
  tikz,%
  pgfplots, %
  xcolor, %
}
\usetikzlibrary{decorations.pathmorphing}
\usetikzlibrary{calc}
\usetikzlibrary{arrows}
\usetikzlibrary{fit}
\usetikzlibrary{backgrounds}
\usetikzlibrary{matrix}
\usetikzlibrary{decorations.pathreplacing} %

\DeclareMathOperator{\adj}{adj}

\definecolor{SPECorange}{rgb}{1.0,.5625,0}
\definecolor{SPECblue}{rgb}{0,0,0.75}
\definecolor{SPECred}{rgb}{0.75,0,0}
\definecolor{SPECgreen}{rgb}{0,0.75,0}
\definecolor{SPECblack}{rgb}{0.75,0.75,0.75}
\definecolor{SPECorangebw}{rgb}{0.2,0.2,0.2}
\definecolor{SPECbluebw}{rgb}{0.8,0.8,0.8}
\definecolor{SPECredbw}{rgb}{0.6,0.6,0.6}
\definecolor{SPECgreenbw}{rgb}{0.4,0.4,0.4}

\usepackage{arydshln}
\usepackage{cite} 

\def\Rc{\raisebox{-0.11cm}[0.1cm][0.0cm]{\mbox{$\Rsh$}}}
\def\rc{\raisebox{0.11cm}[0.1cm][0.0cm]{\mbox{\rotatebox[origin=c]{180}{\reflectbox{\Rc}}}}}

\usepackage{xargs} 
\newcommandx{\tikzrotation}[6][1=black,2=,3=,4=]{
  \node at (#5,#6+.025) [align=center,scale=1.0,color=#1]{$\Rc$}; %
  \node at (#5,#6+.270) [align=center,scale=1.0,color=#1]{$\rc$}; %
  \node at (#5+.1,#6+.1) [align=center,scale=1.0,color=#1]{#2};%
  \node at (#5,#6+0.4) [align=center,scale=1.0,color=#1]{#3};%
  \node at (#5,#6-0.2) [align=center,scale=1.0,color=#1] {#4};
}

\newcommand{\uppertriangular}[3][black]{
  \draw[color=#1] (#2,#3) -- (#2+1.4,#3) -- (#2+1.4,#3+1.4) -- cycle;
}

\newcommand{\shiftthroughrl}[4][black]{
  \path[color=#1,->,out=180,in=0,looseness=0.25] (#2-0.1,#4+0.1) edge
  (#3+0.05,#4+0.1); }

\newcommand{\shiftthroughlr}[4][black]{
  \path[color=#1,->,out=0,in=180,looseness=0.25] (#2+0.05,#4+0.1) edge
  (#3-0.1,#4+0.1); }

\newcommand{\transferbulgelr}[4][black]{ 
  \path[->,out=-45,in=-135,looseness=.25] (#2+0.1,#4+0.2) edge (#3-0.1,#4+0.2);
}

\newcommand{\turnoverrl}[5][black]{
  \path[color=#1,->,out=180,in=0] (#2-0.1,#3+0.1) edge (#4+0.05,#5+0.1);
}

\newcommand{\turnoverlr}[5][black]{
  \path[color=#1,->,out=0,in=180] (#2+0.05,#3+0.1) edge (#4-0.1,#5+0.1);
}

\newcommand{\drawbrace}[5][black]{
  \draw[color=#1,decorate,decoration=brace] (#3+0.05,#4)--(#2-0.05,#4);
  \node[color=#1] at (${0.5}*(#2+#3,#4+#4+0.6)$) [align=center] {#5};
}

\pgfrealjobname{companion_qz}

\usetikzlibrary{calc}
\usepackage{ifpdf}
\ifpdf
\usepackage[pdftex,colorlinks=true,linkcolor=SPECorange,filecolor=SPECgreen,citecolor=SPECblue,pdfpagemode=UseNone]{hyperref}
\else
\usepackage[dvips,colorlinks=true,linkcolor=SPECorange,filecolor=SPECgreen,citecolor=SPECblue,pdfpagemode=UseNone]{hyperref}
\fi

\usepackage{booktabs}
\usepackage{multirow}
\usepackage{showkeys}

\newtheorem{example}[theorem]{\it Example}

\providecommand{\abs}[1]{\left\lvert#1\right\rvert} %
\newcommand{\norm}[1]{\mbox{$\parallel\!#1\!\parallel$}}
\newcommand{\diagg}[1]{\mbox{diag$\left\{#1\right\}$}}

\newcommand{\reals}{\mbox{$\mathbb{R}$}}
\newcommand{\cplxs}{\mbox{$\mathbb{C}$}}
\newcommand{\cn}{\mbox{$\cplxs^{n}$}}
\newcommand{\cnxn}{\mbox{$\cplxs^{n \times n}$}}

\newcommand{\absval}[1]{\mbox{$\mid\!#1\!\mid$}}

\newcounter{mymac@matlab}
\newcommand{\matlab}{MATLAB%
  \ifnum\value{mymac@matlab}<1%
  \textsuperscript{\textregistered}%
  \setcounter{mymac@matlab}{1}%
  \fi%
}

\hypersetup{%
  pdftitle = {Fast and backward stable computation of roots of polynomials, Part II:
    backward error analysis; companion matrix and companion pencil} %
  pdfsubject = {\today}, %
  pdfauthor = {Jared L. Aurentz, Thomas Mach, Leonardo Robol, Raf Vandebril, and David
    S. Watkins}, %
  pdfkeywords = { polynomial, root, companion matrix, companion pencil,
    eigenvalue, Francis algorithm, QR algorithm, QZ algorithm, rotators }
}

\title{Fast and backward stable computation of roots of polynomials, Part II: backward
  error analysis; companion matrix and companion pencil\thanks{
The research was partially supported by 
the Research Council KU Leuven, project  
C14/16/056 (Inverse-free Rational Krylov Methods: Theory and Applications), by
an INdAM/GNCS project, and by the Spanish Ministry
of Economy and Competitiveness, through the Severo Ochoa Programme for Centers
of Excellence in R\&D (SEV-2015-0554).
}}

\author{Jared L.\ Aurentz\footnotemark[2] \and %
  Thomas Mach\footnotemark[5] \and %
  Leonardo Robol\footnotemark[6] \and 
  Raf Vandebril\footnotemark[3] \and %
  David S.\ Watkins\footnotemark[4]}

\begin{document}
\makeatletter
\let\ftype@table\ftype@figure
\makeatother

\maketitle

\renewcommand{\thefootnote}{\fnsymbol{footnote}}

\footnotetext[2]{%
Instituto de Ciencias Matem\'aticas, Universidad Aut\'onoma de Madrid,
 Madrid, Spain
\mbox{(\texttt{JaredAurentz@gmail.com})}.
}

\footnotetext[3]{
  Department of Computer Science, University of Leuven, KU~Leuven, Leuven,  Belgium;
  \mbox{(\texttt{Raf.Vandebril@cs.kuleuven.be})}.}

\footnotetext[6]{
   Istituto di Scienza e Tecnologie dell Informazione `A. Faedo' (ISTI), CNR, Pisa, Italy;
  \mbox{(\texttt{Leonardo.Robol@isti.cnr.it})}.}%

\footnotetext[4]{%
  Department of Mathematics, Washington State University, Pullman, WA
  99164-3113, USA;  \mbox{(\texttt{watkins@math.wsu.edu})}.}

\footnotetext[5]{%
  Department of Mathematics, School of Science and Technology, Nazarbayev
  University, 010000 Astana, Kazakhstan;
  \mbox{(\texttt{thomas.mach@nu.edu.kz})}.  }

\renewcommand{\thefootnote}{\arabic{footnote}}

\date{\today}
\maketitle

\begin{abstract}
This work is a continuation of \emph{Fast and backward stable computation of roots 
of polynomials} by J.L.\ Aurentz, T.\ Mach, R.\
  Vandebril, and D.S.\ Watkins, SIAM Journal on Matrix Analysis and Applications, 36(3):
  942--973, 2015.  In that paper we introduced a companion QR algorithm that finds the
  roots of a polynomial by computing 
  the eigenvalues of the companion matrix in $O(n^{2})$ time using $O(n)$
  memory.  We proved that the method is backward stable.   Here we introduce, as 
  an alternative, a companion QZ algorithm that solves a generalized eigenvalue problem
  for a companion pencil. 
  More importantly, we provide an improved backward error analysis that takes advantage
  of the special structure of the problem.  The improvement is also due, in part, to an improvement
  in the accuracy (in both theory and practice) 
  of the turnover operation, which is the key component of our algorithms.
  We prove that for the companion QR algorithm, 
  the backward error on the polynomial coefficients varies linearly with the norm of the 
  polynomial's vector of coefficients.  Thus the companion QR algorithm has a smaller backward error 
  than the unstructured QR algorithm (used by MATLAB's \texttt{roots} command, for example), 
  for which the backward error on the polynomial coefficients grows quadratically with the norm
  of the coefficient vector.   The companion QZ algorithm has the same favorable backward 
  error as companion QR, provided that the polynomial coefficients are properly scaled.
\end{abstract}

\begin{keywords}
  polynomial, %
  root, %
  companion matrix, %
  companion pencil, %
  eigenvalue, %
  Francis algorithm, %
  QR algorithm, %
  QZ algorithm, %
  core transformation, %
  backward stability
\end{keywords}

\begin{AMS}
  65F15, 
  65H17,
  15A18, 
  65H04
\end{AMS}

\pagestyle{myheadings} \thispagestyle{plain} %
\markboth{J.\ L.\ Aurentz, T.\ Mach, L. Robol, R.\ Vandebril, and D.\ S.\ Watkins}{Fast
  and  stable computation of roots}

\section{Introduction}  
\label{sec:introduction}

Consider the problem of computing the $n$ zeros of a complex polynomial 
\begin{align*}
  p(z) = a_{n}z^{n} + a_{n-1}z^{n-1} + \cdots + a_{1}z + a_{0}, \quad a_{n} \neq 0,\ \ a_{0} \neq 0,
\end{align*}
expressed in terms of the monomial basis.  One way to do this is to 
form the companion matrix
\begin{align}
  A = \begin{bmatrix}
    &        &   & -a_{0}/a_{n} \\
    1 &        &   & -a_{1}/a_{n} \\
    & \ddots &   & \vdots \\
    &        & 1 & -a_{n-1}/a_{n}
  \end{bmatrix},
\label{eq:companionmatrix}
\end{align} 
and compute its eigenvalues.  This is what MATLAB's \texttt{roots} command does.
But \texttt{roots} does not exploit the special structure of the companion
matrix, so it requires $O(n^{2})$ storage (one matrix) and $O(n^{3})$ flops
using Francis's implicitly-shifted $QR$ algorithm \cite{Fra61}.  It is natural
to ask whether we can save space and flops by exploiting the structure.  This
has, in fact, been done by Bini et al.\ \cite{BiBoEi10}, Boito et al.\ \cite{BoEiGe},
Chandrasekaran et al.\ \cite{q232}, and Aurentz et al.\ \cite{AuMaVaWa15}.  
All of the methods proposed in these papers use
the unitary-plus-rank-one structure of the companion matrix to build a special
data structure that brings the storage down to $O(n)$.  Then Francis's
algorithm, operating on the special data structure, has a flop count of
$O(n^{2})$.  Based on the tests in \cite{AuMaVaWa15}, our method appears to be the
fastest of the several that have been proposed. %
Moreover our algorithm is the only one that has been proved to be backward stable.
In this paper we will refer to our method as the \emph{companion QR algorithm}.

In cases where the polynomial has a particularly small leading coefficient, one might hesitate
to use the companion matrix, since division by a tiny $a_{n}$ will result in very
large entries in (\ref{eq:companionmatrix}).  This might adversely affect accuracy.  
An alternative to division by $a_{n}$ is to work with a \emph{companion pencil} 
\begin{align}
\label{eq:clp:top}
  A - \lambda B = \begin{bmatrix}
    0& & & &-a_{0}\\
    1&0& & &-a_{1}\\
    &\ddots&\ddots& &\vdots\\
    & &1&0&-a_{n-2}\\
    & & &1&-a_{n-1}\\
  \end{bmatrix} 
  - \lambda 
  \begin{bmatrix}
    1& & & &\\
    &1& & &\\
    & &\ddots& &\\
    & & &1&\\
    & & & &a_{n}\\
  \end{bmatrix},
\end{align}
which also has the roots of $p$ as its eigenvalues.  More generally we can consider any pencil  of
the form
\begin{align}
\label{eq:cpp}
  V - \lambda W = \left[\begin{array}{cccccc}
      0 &   &        &   &  -v_{1}     \\
      1 &   &        &   &  -v_{2}      \\
        & 1 &        &   &  -v_{3}      \\
        &   & \ddots &   & \vdots  \\
        &   &        & 1 &  -v_{n}
    \end{array}\right]  - \lambda \,\left[\begin{array}{cccccc}
      1 &   &        &   &  w_{1}  \\
        & 1 &        &   &  w_{2}   \\
        &   & \ddots &   & \vdots        \\
        &   &        & 1 & w_{n-1}   \\
        &   &        &   & w_{n}     
    \end{array}\right],
\end{align}
where 
  \begin{align}
    v_{1} &= a_{0},\nonumber\\
    v_{i+1} + w_{i} &= a_{i}, \ \text{for}\
    i=1,\dotsc,n-1,\text{and} \label{eq:def:vw} \\
    w_{n} &= a_{n}. \nonumber
  \end{align} 
One easily checks that all pencils of this form also have the zeros of $p$ as their eigenvalues \cite{AuMaVaWa14b,q573}.  
This generalized eigenvalue problem is in Hessenberg-triangular form and can be solved 
by the Moler-Stewart variant \cite{q361,Wa10,b333} of Francis's algorithm, commonly called the QZ algorithm. 
In this paper we introduce a generalization of the method of \cite{AuMaVaWa15} to matrix pencils, which  
we will call the \emph{companion QZ algorithm}.  
This is a straightforward exercise, and it is not the main point of this publication.

This paper is really about backward error analysis.  After the publication of \cite{AuMaVaWa15} we realized that
the analysis in that paper was not quite right.  There is a factor that is constant in exact arithmetic and that we treated
as a constant.  We should have taken into account the backward error in that factor.  
In this paper we take the opportunity to repair that error.  Moreover we 
will improve the analysis by taking into account the structure
of the backward error that follows from the structure of the companion matrix or pencil.   

Our intuition told us that there would be cases 
where the companion QR method fails but companion QZ succeeds:  
just take a polynomial with a tiny leading coefficient!  
Once we had implemented companion QZ, we looked for examples 
of this kind, but to our surprise we were not able to find any.  
The companion QR method is much more robust than we had realized!

This discovery led us to take a closer look at the backward error analysis
and try to explain why the companion QR algorithm works so well.  
This paper is the result of that investigation.  We prove that the 
companion matrix method is just as good as the companion pencil method.  Both methods have significantly 
better backward errors than a method like MATLAB's \texttt{roots} that computes the eigenvalues of the companion matrix 
without exploiting the special structure.  This is the main message of this paper.  

The paper is organized as follows.  Section~\ref{sec:background} briefly discusses previous work in 
this area.  The memory-efficient  $O(n)$
factorization of a companion pencil $V - \lambda W$  of the form (\ref{eq:cpp}) 
is introduced in Section~\ref{sec:core}, together with some necessary background information. 
In Section~\ref{sec:companion:QZ:algorithm} we introduce the companion QZ algorithm 
and demonstrate its $O(n^{2})$ performance.  

The heart of the paper is the backward error analysis in Section~\ref{sec:stab}.
The main results are as follows:  Let $p$ be a monic polynomial with coefficient
vector
$a = \left[\begin{array}{cccc} a_{0} & \cdots & a_{n-1} & 1 \end{array}\right]$.
Suppose we compute the zeros of $p$ by some method, and let $\hat{p}$, with
coefficient vector $\hat{a}$, be the monic polynomial that has the computed
roots as its exact zeros.  The \emph{absolute normwise backward error on the
  coefficients} is $\norm{a - \hat{a}}$.  If our companion QR method is used to
do the computation, the backward error satisfies
$\norm{a - \hat{a}} \lesssim u\,\norm{a}$, where $u$ is the unit
roundoff and $\lesssim$ means less than or equal up to a modest
 multiplicative constant depending on $n$ as a low-degree 
 polynomial.\footnote{This result is valid for the improved turnover that is
  introduced \S~\ref{subsec:newturn}.  If the old turnover is used, we get
  $\norm{a - \hat{a}} \lesssim u\norm{a}^{2}$.  We reported this in
  \cite{AuMaVaWa15}, but for a valid proof see \cite{TW683}.  Even with the old
  turnover we can get a better result as follows.  Instead of comparing $p$ with
  $\hat{p}$, we can compare with $\gamma\hat{p}$, where $\gamma$ is chosen so
  that $\norm{a - \gamma\hat{a}}$ is minimized.  Then we get
  $\norm{a - \gamma\hat{a}} \lesssim u\,\norm{a}$, as we have shown in
  \cite{TW683}.}  This is an optimal result, and it is better than can be
achieved by the unstructured Francis algorithm applied to the companion matrix.
The latter gives only $\norm{a - \hat{a}} \lesssim u\, \norm{a}^{2}$.  If the
companion \emph{pencil} is used, we do not get the optimal result unless we
apply the companion QZ algorithm (or any stable algorithm) to a rescaled
polynomial $p/\norm{a}$.  If we do this, we get the optimal result
$\norm{a - \hat{a}} \lesssim u\,\norm{a}$.

  


\section{Earlier  work}\label{sec:background}
There are many ways \cite{q036} to compute roots of polynomials. Here we focus on companion 
matrix and pencil methods.   
Computing roots of polynomials in the monomial basis via the companion matrix has been the
subject of study of several research teams. See \cite{AuMaVaWa15} for a summary.  

There is, to our knowledge, only one article by Boito, Eidelman, and
Gemignani \cite{BoEiGe14} that presents a structured QZ algorithm for computing roots of
polynomials in the monomial basis. The authors consider a matrix pencil, say $(V,W)$,
where both $V$ and $W$ are of unitary-plus-rank-one form, $V$ is Hessenberg and $W$ is
upper triangular. Both matrices are represented efficiently by a quasiseparable
representation \cite{BoEiGe14}. 
To counter the effects of roundoff errors, some redundant quasiseparable generators
to represent the unitary part are created; a compression algorithm to reduce the number of
generators to a minimal set is presented.  The computational cost of each structured QZ
iteration is $O(n)$.  
A double shift version of this algorithm is presented by Boito, Eidelman, and Gemignani in \cite{BoEiGe16}.

The companion pencil \eqref{eq:clp:top} is the most frequently appearing one in
the literature, but, there is a wide variety of comparable matrix pencils with
the same eigenvalues \cite{EasKimShaVan14,AuMaVaWa15b,q573}, 
many of which are highly structured.  In this article we will focus on companion
pencils of the form \eqref{eq:cpp}.


\section{Core transformations and factoring  companion matrices}
\label{sec:core}
Core transformations will be used throughout the paper and are the
building blocks for a fast algorithm and an efficient representation of the
companion pencil.


\subsection{Core transformations}
\label{sec:core:transfo}
A nonsingular matrix $G_{i}$ identical to the identity matrix except for a $2 \times 2$
submatrix in position $(i:i+1,i:i+1)$ is called a \emph{core transformation}. The subscript
$i$ refers to the position of the diagonal block $(i:i+1,i:i+1)$ called the \textit{active part} of
the core transformation. Core transformations $G_{i}$ and $G_{j}$
 commute if $\abs{i-j}>1$.

 In previous work \cite{AuVaWa13,AuVaWa14} the authors have used
 non-unitary core transformations, 
but here we will only use unitary core
 transformations.  Thus, in this paper, the term \emph{core
   transformation} will mean \emph{unitary} core transformation; the
 active part could be a rotator or a reflector, for example.

To avoid excessive index usage, and to ease the understanding of the interaction of core
transformations, we depict them as
$\begin{smallmatrix} \Rc\\[0.7ex] \rc \end{smallmatrix}$, where the tiny arrows
pinpoint the active part. For example, every unitary upper Hessenberg matrix $Q$ can be factored 
as the product of $n-1$ core transformations in a \textit{descending} order $Q=G_{1} G_{2} \cdots G_{n-1}$. 
Such a \textit{descending sequence} of core transformations is represented pictorially by 
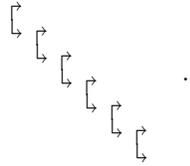
\begin{equation*}
  \begin{tikzpicture}[scale=1.66,y=-1cm,inner xsep=0cm,baseline={(current bounding box.center)}]
    \pgfmathsetmacro{\xoff}{0.0}
    \pgfmathsetmacro{\yoff}{0.0}
    \foreach \j in {0,0.2,...,1.0} {%
      \tikzrotation{\j+\xoff}{\j+\yoff}%
    }%
\end{tikzpicture}\quad .
\end{equation*}

All of the algorithms in this paper are described in terms of core transformations and two operations: the
\emph{fusion} and the \emph{turnover}.

\paragraph{Fusion} The product of two unitary core transformations $G_{i}$
and $H_{i}$ is again a unitary core transformation. Pictorially we can write this as
\begin{align*}
  \begin{matrix}\Rc & \Rc\\\rc& \rc\end{matrix}\ 
= \ \begin{matrix}\Rc\\ \rc\end{matrix}.
\end{align*} 

\paragraph{Turnover} The product of three core transformations
$F_{i} G_{i+1} H_{i}$ is an essentially $3\times 3$ unitary matrix that can be
factored also 
as $F_{i+1} G_{i} H_{i+1}$, depicted as
\begin{align*}
  \begin{matrix}
    \Rc & &\Rc\\
    \rc &\Rc &\rc\\
    &\rc &\\
  \end{matrix} \ = 
  \begin{bmatrix}
    \times & \times &\times\\
    \times & \times &\times\\
    \times & \times &\times\\
  \end{bmatrix} = \
  \begin{matrix}
    &\Rc&\\
    \Rc &\rc &\Rc\\
    \rc& &\rc\\
  \end{matrix}\,.
\end{align*} %

\paragraph{Pictorial action of a turnover and fusion}
We will see, when describing the algorithms, that
there are typically one or more core transformations not fitting the pattern (called the \emph{misfit(s)}), that
need to be moved around by executing turnovers and similarities. To describe clearly the movement of the
misfit we use the following pictorial description:
\begin{center}
  \begin{tikzpicture}[scale=1.66,y=-1cm]
    \tikzrotation{-0.2}{0.2}
    \tikzrotation{0.0}{0.0}
    \tikzrotation{0.2}{0.2}
    \tikzrotation{0.4}{0.0}
    \node at (-.2,0.75) [align=center] {$B_{2}$};
    \node at (0.0,0.55) [align=center] {$G_{1}$};
    \node at (0.2,0.75) [align=center] {$G_{2}$};
    \node at (0.4,0.55) [align=center] {$B_{1}$};  
    \turnoverrl{0.4}{0.0}{-0.2}{0.2}
    \node[above] at (0.6,0.7) {,};
  \end{tikzpicture}
\end{center} where $B_{1}$ is the misfit before the turnover and
$B_{2}$ after the turnover. The core transformations $G_{1}$ and $G_{2}$ are involved in the
turnover and change once $B_{2}$ is created. The picture is mathematically equivalent to
$G_1G_2 B_1 = B_2 \hat{G}_1 \hat{G}_2$. 
Other possible turnovers are
\begin{center}
  \begin{tikzpicture}[scale=1.66,y=-1cm]
    \tikzrotation{-0.2}{0.2}
    \tikzrotation{0.0}{0.0}
    \tikzrotation{0.2}{0.2}
    \tikzrotation{0.4}{0.0}
    \turnoverlr{-0.2}{0.2}{0.4}{0.0}

    \node[above right] at (0.5,0.4) {,};
    
    \tikzrotation{0.8}{0.0}
    \tikzrotation{1.0}{0.2}
    \tikzrotation{1.2}{0.0}
    \tikzrotation{1.4}{0.2}
    \turnoverlr{0.8}{0.0}{1.4}{0.2}
    
    \node[above right] at (1.5,0.4) {, and};
    
    \tikzrotation{2.4}{0.0}
    \tikzrotation{2.6}{0.2}
    \tikzrotation{2.8}{0.0}
    \tikzrotation{3.0}{0.2}
    \turnoverrl{3.0}{0.2}{2.4}{0.0}
    \node[above right] at (3.1,0.4) {.};
  \end{tikzpicture}
\end{center}
Directly at the start and at the end of each QZ (and QR) iteration, a misfit is
fused with another core transformation, so that it vanishes.  We will describe this
pictorially as
\begin{center}
  \begin{tikzpicture}[scale=1.66,y=-1cm]
    \tikzrotation[black][][$G$]{-0.2}{0.2}
    \tikzrotation[black][][$B$]{0.2}{0.2}
    \path[->,out=180,in=0] (0.10,0.3) edge (-0.15,0.3);     
    \node [above] at (0.75,0.5) {or};
    \tikzrotation[black][][$B$]{1.3}{0.2}
    \tikzrotation[black][][$G$]{1.7}{0.2}
    \path[<-,out=180,in=0] (1.60,0.3) edge (1.35,0.3);     
    \node[above] at (1.9,0.5) {,};
  \end{tikzpicture}
\end{center}
where $B$ is to be fused with $G$.

\subsection{A factorization of the companion pencil}
\label{subsec:factor}
The pencil matrices $V$ and $W$ from \eqref{eq:cpp} are both unitary-plus-rank one, $V$
is upper Hessenberg and $W$ is upper triangular. 
We store $V$ in QR decomposed form:  $V=QR$, where $Q$ is unitary and upper Hessenberg,
and $R$ is upper triangular and unitary-plus-rank-one.  In fact
\begin{align}
  Q = \left[\begin{array}{ccccc}
      0 &   &        &   &  1     \\
      1 &   &        &   &  0       \\
      & 1 &        &   &  0        \\
      &   & \ddots &  & \vdots   \\
      & & & 1 & 0 
    \end{array}\right] \quad\mbox{and}\quad 
  R = \left[\begin{array}{ccccc}
      1 &   &        &   &  -v_{2}        \\
      & 1 &        &   &  -v_{3}        \\
      &   & \ddots &   & \vdots   \\
      &   &        & 1 & -v_{n}   \\
      & & & & -v_{1} 
    \end{array}\right].
\label{eq:qrfactv}
\end{align} 
We need efficient representations of $Q$, $R$, and $W$.  
$Q$ is easy;  it is the product of $n-1$ core transformations:
$Q=Q_{1}\dotsm Q_{n-1}$, with $Q_{i}(i:i+1,i:i+1)=\left[\begin{smallmatrix} 0 &
    1\\ 1 & 0\end{smallmatrix}\right]$. 
    
\paragraph{Factoring an upper triangular unitary-plus-rank-one matrix}  
The matrices $R$ \eqref{eq:qrfactv} and $W$ \eqref{eq:cpp} have exactly the same 
structure and can be factored in the same way.   This factorization was introduced and
studied in detail in \cite{AuMaVaWa15}, so we will just give a brief description here.  
We focus on $R$.  It turns out that 
for this factorization we need to add a bit of room by adjoining a row and column.  
Let 
\begin{align}\label{eq:bigr}
\underline{R} = \left[\begin{array}{ccccc|c}
      1 &   &        &   &  -v_{2}  & 0      \\
      & 1 &        &   &  -v_{3}  & 0      \\
      &   & \ddots &   & \vdots   & \vdots \\
      &   &        & 1 & -v_{n}  & 0  \\
      & & & & -v_{1} & 1 \\ \hline & & & & 0 & 0
    \end{array}\right].
\end{align}
This is just $R$ with a zero row and  nearly zero column added.  The $1$ in the last 
column ensures that $\underline{R}$ is unitary-plus-rank-one:  $\underline{R} = Y_{n} + \underline{z}e_{n}^{T}$,
where 
\begin{align}\label{eq:uxform}
Y_{n} = 
\left[\begin{array}{ccccc|c}
1 &   &        &   &    &        \\
  & 1 &        &   &    &       \\
  &   & \ddots &   &   & \vdots \\
  &   &        & 1 &   &   \\ 
 &   &        &   &   0  & 1     \\ \hline 
  &   &       &   &   1  & 0
\end{array}\right] \quad\mbox{and}\quad 
\underline{z} = -\left[\begin{array}{c} 
v_{2} \\ v_{3} \\ \vdots \\ v_{n} \\ v_{1} \\ \hline 
1
\end{array}\right].
\end{align} 

Let $C_{1},\ \dotsc,\ C_{n}$ be core transformations such that $C \underline{z} = C_{1}
\dotsm C_{n} \underline{z} = \alpha e_{1}$, where $\absval{\alpha} = \norm{\underline{z}}_{2}$. 
Let $B$ be the unitary Hessenberg matrix $B=C Y_{n}$.  Clearly 
$B=B_{1}\dotsm B_{n}$, where $B_{i}=C_{i}$
for $i=1,\dotsc,n-1$, and $B_{n}=C_{n}Y_{n}$.
This gives us a factorization of $\underline{R}$ as
\begin{align}
  \underline{R} = C_{n}^{*}\dotsm C_{1}^{*}( B_{1}\dotsm B_{n} + \alpha\,e_{1}\underline{y}^{T}) = C^{*}(B +
  \alpha\,e_{1}\underline{y}^{T}), \label{eq:factorize:upper:triangular}
\end{align} 
with $\underline{y}^{T}=e_{n}^{T} \in \reals^{n+1}$.   In the course of our algorithm, the core transformations
$B_{i}$, $C_{i}$, and the vector $\underline{y}$ will be modified repeatedly, but the form 
\eqref{eq:factorize:upper:triangular} for $\underline{R}$ will be preserved.  $\underline{R}$ remains upper
triangular with its last row  equal to zero.   The theory that supports these
claims can be found in \cite[\S~4]{AuMaVaWa15}.   Notice that the core transformations $B_{n}$ and 
$C_{n}$ both make use of row and column $n+1$.  Had we not added a row and column, this factorization 
would not have been possible.

Multiplying \eqref{eq:factorize:upper:triangular} on the left by $e_{n+1}^{T}$, we find that
$0 = e_{n+1}^{T}\underline{R} = e_{n+1}^{T}C^{*}B + 
\alpha\,e_{n+1}^{T}C^{*}e_{1}\underline{y}^{T}$, so \cite[Thm.~4.6]{AuMaVaWa15}
\begin{equation}\label{eq:yrecover}
\alpha\,\underline{y}^{T}=-(e_{n+1}^{T}C^{*}e_{1})^{-1}e_{n+1}^{T}C^{*}B.
\end{equation}
This equation demonstrates that the information about $\alpha\,\underline{y}^{T}$, 
which determines the rank-one part, is encoded
in the core transformations.   This means that we will be able to develop an algorithm that does not keep 
track of $\underline{y}$; the rank-one part can be simply ignored.   
If at any time we should need $\underline{y}$ or some part of $\underline{y}$, 
we could recover it from $C$ and $B$ using \eqref{eq:yrecover}.
However, as we show in 
\cite[\S~4]{AuMaVaWa15}, it turns out that we never need to use \eqref{eq:yrecover} in practice.  

Let $P = I_{(n+1) \times n}$,   the $(n+1) \times n$ matrix obtained by deleting the last column
from the $(n+1) \times (n+1)$ identity matrix.  Then $R = P^{T}\underline{R}P$, so our factored form of $R$ is 
\begin{align} 
\label{eq:fac:uptri:2}
  R = P^{T}C_{n}^{*}\dotsm C_{1}^{*}( B_{1}\dotsm B_{n} + \alpha\,e_{1}\underline{y}^{T})P. 
\end{align} 

The matrices $P$ and $P^{T}$ are included just so that the dimensions of $R$ come out right.
They play no active role in the algorithm, and we will mostly ignore them.  
Pictorially, for $n=8$, $\underline{R}$ (and hence also $R$) can be represented as 
\begin{center}
  \begin{tikzpicture}[scale=1.66,y=-1cm,inner xsep=0cm]
    \foreach \j in {0,0.2,...,1.4} {
      \tikzrotation{3.0-\j}{\j}
    }
    \drawbrace{1.6}{3.0}{1.7}{$C^{*}=C_{n}^{*}\dotsm C_{1}^{*}$}
    \draw (3.275,-0.2) -- (3.2,-0.2) -- (3.2,1.8) -- (3.275,1.8);
    \foreach \j in {0,0.2,...,1.4} {
      \tikzrotation{3.4+\j}{\j}
    } 
    \drawbrace{3.4}{4.8}{1.7}{$B=B_{1}\dotsm B_{n}$}
    \node[above] at (5.0,1.0) {$+$};
    \pgfmathsetmacro{\xoffe}{5.2}
    \draw (\xoffe+0.05,-0.1) -- (\xoffe+0.0,-0.1) -- (\xoffe+0.0,1.7) -- (\xoffe+0.05,1.7);    
    \foreach \j/\jj in {0/$1$,0.2/$0$,0.4/$0$,
      0.6/$0$,0.8/$0$,1.0/$0$,
      1.2/$0$,1.4/$0$,1.6/$0$} {
      \node at (\xoffe+0.2,\j) [align=center] {\jj};
    }
    \node at (\xoffe+0.2,1.85) [align=center] {$e_{1}$};
    \draw (\xoffe+0.35,-0.1) -- (\xoffe+0.4,-0.1) -- (\xoffe+0.4,1.7) -- (\xoffe+0.35,1.7);
    \pgfmathsetmacro{\xoff}{5.8}
    \draw (\xoff+0.05,-0.1) -- (\xoff,-0.1) -- (\xoff,0.1) -- (\xoff+0.05,0.1);%
     \foreach \j/\jj in {%
       0/{$\times$}, 0.2/{$\times$}, 0.4/{$\times$},
       0.6/{$\times$}, 0.8/{$\times$}, 1.0/{$\times$},
       1.2/{$\times$}, 1.4/{$\times$}, 1.6/{$\times$}} {%
       \node at (6.0+\j,0) [align=center] {\jj};%
     }%

    \draw (\xoff+1.95,-0.1) -- (\xoff+2.0,-0.1) -- (\xoff+2.0,0.1) -- (\xoff+1.95,0.1);%

    \drawbrace{\xoff}{\xoff+2.0}{0.2}{$\alpha\,\underline{y}^{T}$}
    \draw (\xoff+2.125,-0.2) -- (\xoff+2.2,-0.2) -- (\xoff+2.2,1.8) -- (\xoff+2.125,1.8);
    \drawbrace{1.6}{\xoff+2.2}{2.2}{$\underline{R}$}
    \node[above] at (\xoff+2.4,1.0) {.};
  \end{tikzpicture}
\end{center}
Since we can ignore $\alpha\,\underline{y}^{T}$, we see that $R$ is represented by two
sequences of core transformations,  $C$ and $B$.   
Hence  $A = QR$ is represented by three sequences of core transformations.  

The matrix $W$ admits a factorization of the same form as 
\eqref{eq:fac:uptri:2}
$$  W = P^{T}C_{W}^{*}(B_{W} + \alpha_{W}e_{1}\underline{y}_W^{T})P,$$
with $C_{W}[w_{1}, \ldots, w_{n}, 1]^{T} = \alpha_{W}e_{1}$. 

Thus $W$ is also represented by two sequences of core transformations.  Altogether the pencil
$(V,W)$ is represented by five sequences of core transformations.

\subsection{Core transformations and upper triangular matrices}
\label{sec:ut}
\label{para:thru:upper:triangular}
In the next section we will show how to compute the eigenvalues of the
matrix pencil $(V,W)$ via the QZ algorithm as described by Vandebril and Watkins \cite{VaWa12}.
An important operation is to refactor the product of an  upper triangular matrix times a
core transformation $RG_{i}$ as the product of a core transformation times an upper
triangular matrix\footnote{We assume that eigenvalues at zero or infinity are deflated beforehand, so the involved upper triangular matrices are nonsingular.} $\hat{G}_{i} \hat{R}$. The other way proceeds similarly.
Considering dense matrices we get pictorially
\begin{center}
  \begin{tikzpicture}
    \node at (0,0) 
    { $
      \begin{bmatrix}
        \times & \times &\times &\times &\times \\
        & \times &\times &\times &\times \\
        &  &\times &\times &\times \\
        &  & &\times &\times \\
        &  & & &\times \\
      \end{bmatrix} \
      \begin{matrix}\\ \Rc\\\rc \\ \\ \\\end{matrix} \ \ = \ \begin{bmatrix}
        \times & \times &\times &\times &\times \\
        & \times &\times &\times &\times \\
        & \times &\times &\times &\times \\
        &  & &\times &\times \\
        &  & & &\times \\
      \end{bmatrix} \ = \ \
      \begin{matrix}\\ \Rc\\\rc \\ \\ \\\end{matrix}\ 
      \begin{bmatrix}
        \times & \times &\times &\times &\times \\
        & \times &\times &\times &\times \\
        &  &\times &\times &\times \\
        &  & &\times &\times \\
        &  & & &\times \\
      \end{bmatrix}.  $ };
  \end{tikzpicture}
\end{center}
 Applying a nontrivial  core transformation from the right on the upper triangular matrix creates
a non-zero subdiagonal, which can be removed by
pulling out a nontrivial core transformation from the left.

In our case the upper triangular matrix is represented in a data-sparse way 
\eqref{eq:factorize:upper:triangular}. Instead of explicitly creating the upper triangular
matrix to interchange the order of the upper triangular matrix and a core transformation
we operate on the factored form directly.
There are two versions, \textit{passing a core transformation} from right to  left and
vice versa. 

Consider an arbitrary upper triangular unitary-plus-rank-one matrix, say,
$R=P^{T}C^{*}(B+\alpha\,e_{1}\underline{y}^{T})P$.
Let $G_{i}$ ($1 \leq i \leq n-1$) be the core transformation we want to move from right to left.  
We have $PG_{i} = \underline{G}_{i}P$, where $\underline{G}_{i}$ is the $(n+1)\times (n+1)$ version of $G_{i}$. 
We have $RG_{i}=P^{T}C^{*} (B+\alpha\,e_{1}\underline{y}^{T})\underline{G}_{i}P
=P^{T}C^{*} (B\underline{G}_{i}+\alpha\,e_{1}\underline{y}^{T}\underline{G}_{i})P$, so
$\underline{y}^{T}$ changes into $\hat{\underline{y}}^{T}=\underline{y}^{T}\underline{G}_{i}$. 
The next step is to execute a turnover
$B_{i}B_{i+1}\underline{G}_{i}=\tilde{\underline{G}}_{i+1}\hat{B}_{i}\hat{B}_{i+1}$.  We now have $RG_{i} = 
P^{T}C^{*}(\tilde{\underline{G}}_{i+1}\hat{B} + \alpha\,e_{1}\hat{\underline{y}}_{1}^{T})P = 
P^{T}C^{*}\tilde{\underline{G}}_{i+1}(\hat{B} + \alpha\,e_{1}\hat{\underline{y}}_{1}^{T})P$.  
In the last step we have used the fact that $\tilde{G}_{i+1}e_{1}=e_{1}$ because $i+1 > 1$.  
To complete the procedure we just need to do one more turnover, in which $\tilde{\underline{G}}_{i+1}$ interacts with $C_{i}^{*}$ and 
$C_{i+1}^{*}$.   Specifically 
$C_{i+1}^{*}C^{*}_{i}\tilde{\underline{G}}_{i+1}=\hat{\underline{G}}_{i}\hat{C}^{*}_{i+1}\hat{C}^{*}_{i}$, resulting in  
$RG_{i} = P^{T}\hat{\underline{G}}_{i}\hat{C}^{*}(\hat{B} + \alpha\,e_{1}\hat{\underline{y}}^{T})P$.  Finally
we have $P^{T}\hat{\underline{G}}_{i} = \hat{G}_{i}P^{T}$,  where $\hat{G}_{i}$ is the $n\times n$ version of 
$\hat{\underline{G}}_{i}$.  Here it is important that $i \leq n-1$. The final result is 
$RG_{i} = \hat{G}_{i}P^{T}\hat{C}^{*}(\hat{B} + \alpha\,e_{1}\hat{\underline{y}}^{T})P = \hat{G}_{i}\hat{R}$.  

The total computational effort required for the task is just two turnovers.  The operation 
$\hat{\underline{y}}^{T} = \underline{y}^{T}\underline{G}_{i}$ is not 
actually performed, as we do not keep track of $\underline{y}$.   

Pictorially for $n=8$ and $i=1$, we have
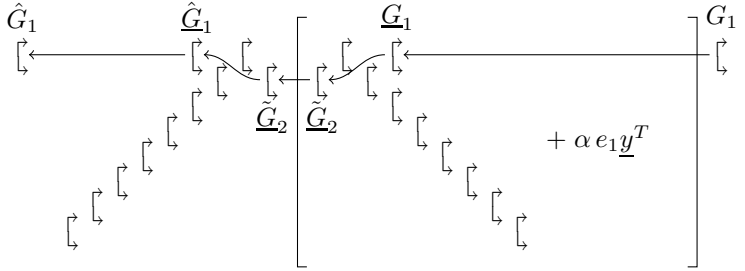
\begin{center}
  \begin{tikzpicture}[scale=1.66,y=-1cm]
    \tikzrotation[black][][][$G_{1}$]{8.4}{0}   
    \shiftthroughrl{8.4}{5.8}{0}    
    
    \tikzrotation[black][][][$\underline{G}_{1}$]{5.8}{0}   
    \turnoverrl{5.8}{0.0}{5.2}{0.2}
    
    \tikzrotation[black][][$\tilde{\underline{G}}_{2}$]{5.2}{0.2}   
    \shiftthroughrl{5.2}{4.8}{0.2}    
    
    \tikzrotation[black][][$\tilde{\underline{G}}_{2}$]{4.8}{0.2}   
    \turnoverrl{4.8}{0.2}{4.2}{0.0}
    
    \tikzrotation[black][][][$\hat{\underline{G}}_{1}$]{4.2}{0.0}   
    
    \shiftthroughrl{4.2}{2.8}{0}
    \tikzrotation[black][][][$\hat{G}_{1}$]{2.8}{0}
    
    \foreach \j in {0,0.2,...,1.4} {
      \tikzrotation{4.6-\j}{\j}
    }      
    \draw (5.075,-0.2) -- (5.0,-0.2) -- (5.0,1.8) -- (5.075,1.8);
    \foreach \j in {0,0.2,...,1.4} {
      \tikzrotation{5.4+\j}{\j}
    }
    \node[above left] at (7.9,1.0) {$+\; \alpha\,e_{1}\underline{y}^{T}$};
    \draw (8.125,-0.2) -- (8.2,-0.2) -- (8.2,1.8) -- (8.125,1.8);
    
    \node[above] at (8.6,1.0) {,};
  \end{tikzpicture}
\end{center}
where we have not depicted $P$ or $P^{T}$, and we have ignored the action on $\underline{y}^{T}$.

Notice that when we apply a core transformation $G_{n-1}$, we temporarily create $\tilde{\underline{G}}_{n}$, which makes
use of row/column $n+1$.  Here we see that the existence of an extra row/column is crucial to the
functioning of the algorithm.  

We can pass a core transformation from left to right through $R$ simply by reversing the above procedure.
We omit the details.

It is clear now that one can move a single core transformation through a factored upper
triangular matrix in either direction by executing only two turnovers. From now on, to ease the notation, 
we will depict our Hessenberg-triangular pencil in a simpler format:
\begin{equation}
  \begin{tikzpicture}[baseline={([yshift=.4cm]current bounding box.center)},scale=1.66,y=-1cm]
    \foreach \j in {0.0,0.2,0.4,...,1.2} {
      \tikzrotation{\j-0.2}{\j}
    }      
    \drawbrace{-0.2}{1.1}{1.7}{$Q$}

    \uppertriangular{1.6}{0.0}
    \node[above] at (2.5,.6) {$R$};
    \drawbrace{1.6}{3.0}{1.7}{$P^{T}C^{*}(B+\alpha\,e_{1}\underline{y}^{T})P$}
    %
    \drawbrace{-0.3}{3.1}{2.2}{$V$}
    %
    \node[above] at (3.4,1.1) {,};      
    \uppertriangular{4.0}{0.0}
    \node[above] at (4.9,.6) {$W$};
    \node[above] at (5.9,1.1) {,};    
    \drawbrace{4.0}{5.4}{1.7}{$P^{T}C_W^{*}(B_W+\alpha_{W}e_{1}\underline{y}_W^{T})P$}
    %
  \end{tikzpicture}
  \label{eq:tria}
\end{equation}
where we have replaced each upper triangular factor by a triangle. With this description,
we can immediately apply the algorithms from Vandebril and Watkins \cite{VaWa12}. For
completeness, however, we will redescribe the flow of a single 
QZ step. 

To facilitate the theoretical description we consider the product 
 $S = RW^{-1}$, which is another upper triangular matrix through which we need to pass
 core transformations. Moving a core transformation from the right to the left 
has two stages:
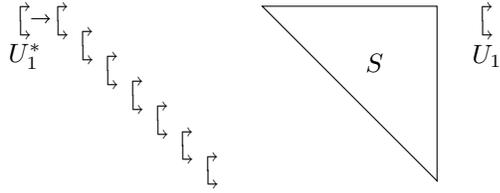
\begin{equation}\label{eq:passthroughs}
  \begin{tikzpicture}[baseline={(current bounding box.center)},scale=1.66,y=-1cm]
    \tikzrotation[black][][$X_i$]{.8}{0.6}
    \tikzrotation[black][][$Z_i$]{3.4}{0.6}
    \tikzrotation[black][][$U_i$]{5.4}{0.6}
    \shiftthroughrl{5.4}{3.4}{0.6} 
    \shiftthroughrl{3.4}{.8}{0.6} 
    \uppertriangular{1.2}{0.0}
    \node[above] at (2.1,.6) {$R$};  


    \uppertriangular{3.6}{0.0} 
    \node[above] at (4.6,.6) {$W^{-1}$};  
    \node[above] at (5.6,.7) {.};    
  \end{tikzpicture}
\end{equation}
Since we do not wish to invert $W$, we do not literally execute the operation depicted on 
the right of \eqref{eq:passthroughs}.  Instead we do the equivalent operation
\begin{center}
  \begin{tikzpicture}[baseline={(current bounding box.center)},scale=1.66,y=-1cm]
    \tikzrotation[black][][$U_i^*$]{3.4}{0.6}
    \tikzrotation[black][][$Z_i^*$]{5.4}{0.6}
    \shiftthroughlr{3.4}{5.4}{0.6} 


    \uppertriangular{3.6}{0.0} 
    \node[above] at (4.6,.6) {$W$};  
    \node[above] at (5.6,.7) {.};    
  \end{tikzpicture}
\end{center}
In fact, the latter operation remains valid, even in the case when $W$ is singular. We
 consider $W^{-1}$ to simplify the description in the next section. 

Since both $R$ and $W$ are unitary-plus-rank-one matrices stored in the factored
form described in Section~\ref{sec:core}, each of the two stages costs two turnovers.
Thus the computational cost of passing a core transformation through $S$ is four turnovers.

\section{The companion QZ algorithm}
\label{sec:companion:QZ:algorithm}
We have implemented both single-shift and double-shift companion QZ algorithms.
For simplicity we will describe only the single-shift case, as the double-shift
iteration is a straightforward extension; we refer to Aurentz et al.\ and
Vandebril and Watkins \cite{AuMaVaWa15,VaWa12}.  The companion QZ algorithm
is easily described by viewing it as a version of the companion QR algorithm
applied to the matrix $VW^{-1}$.  
Clearly
$$VW^{-1} = QRW^{-1} = QS,$$
where $S = RW^{-1}$ is upper triangular. As discussed in Section~\ref{sec:ut} the matrix $W$ does not
need to be invertible and one could as well work on the pair ($R$,$W$) instead which is
equivalent to the classical description of QZ algorithms. There are no issues when
considering  $W$ and we can
rely on standard techniques \cite{b333,Wa75,Wa00}.



Pictorially 
\begin{center}
  \begin{tikzpicture}[scale=1.66,y=-1cm]
    \node[above] at (-1.2,1.2) {$VW^{-1}=$};
    \foreach \j in {0.0,0.2,0.4,...,1.2} {
      \tikzrotation{\j-0.4}{\j}
    }      
    \drawbrace{-0.4}{0.9}{1.7}{$Q$}
    \uppertriangular{1.2}{0.0}
    \drawbrace{1.2}{2.6}{1.7}{$S=RW^{-1}$}
    %
    \node[above] at (3.1,.8) {.};      
  \end{tikzpicture}
\end{center}

To begin the iteration we select a suitable shift $\mu$ and compute $q = (V - \mu W) e_{1}$. 
Only the first two entries of $q$ are nonzero, so we can construct a core transformation $U_{1}$ such that 
$U_{1}^{*}q = \alpha e_{1}$ for some $\alpha$.  Our first modification to $VW^{-1}$ is to apply a similarity transformation
by $U_{1}$:
\begin{center}
  \begin{tikzpicture}[scale=1.66,y=-1cm]
    \tikzrotation[black][][$U_{1}^{*}$]{-.7}{0}
    \tikzrotation[black][][$U_{1}$]{3.0}{0}
     \draw[->] (-.65,0.1) -- (-.5,0.1);  
    \foreach \j in {0.0,0.2,0.4,...,1.2} {
      \tikzrotation{\j-0.4}{\j}
    }      
    \uppertriangular{1.2}{0.0}
    \node[above] at (2.1,.6) {$S$};
    %
    \node[above] at (3.4,.8) {.};      
  \end{tikzpicture}
\end{center}

We can immediately fuse $U_{1}^{*}$ with $Q_{1}$ to make a new $Q_{1}$.  (To keep the notation under control, we 
do not give the modified $Q_{1}$ a new name; we simply call it $Q_{1}$.)  We can also pass $U_{1}$ through $S$
to obtain 
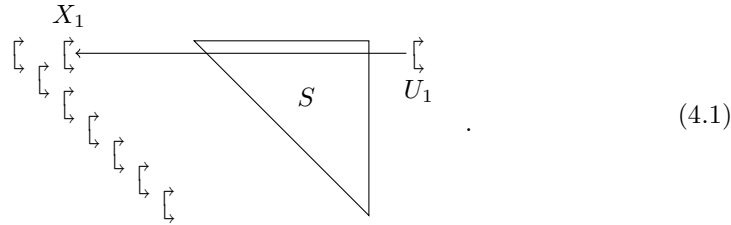
\begin{equation}\label{eq:first_bulge}
  \begin{tikzpicture}[baseline={(current bounding box.center)},scale=1.66,y=-1cm]
    \foreach \j in {0.0,0.2,0.4,...,1.2} {
      \tikzrotation{\j}{\j}
    }  
    \uppertriangular{1.4}{0.0}
    \node[above] at (2.3,.6) {$S$};
    \node[above] at (3.6,.8) {.};      

    \pgfmathsetmacro{\xuone}{-.3}
    \pgfmathsetmacro{\xutwo}{0.4}
    \pgfmathsetmacro{\xuthree}{3.2}
    \pgfmathsetmacro{\xufour}{3.9}
    \pgfmathsetmacro{\xufive}{5.9}

    \tikzrotation{\xutwo}{0}
    \node at (\xutwo,-0.2) [align=center] {$X_{1}$};
    \shiftthroughrl{\xuthree}{\xutwo}{0.0}
    \tikzrotation[black][][$U_{1}$]{3.2}{0}

  \end{tikzpicture}
\end{equation}
The details of passing a core transformation through $S$ where described in Section~\ref{sec:ut}.

If we were to multiply the factors together, we would find that the matrix is no longer upper Hessenberg;
there is a bulge in the Hessenberg form caused by a nonzero entry in position $(3,1)$.   The standard Francis
algorithm chases the bulge until it disappears off the bottom of the matrix.  In our current setting we do not see
a bulge.  Instead we see an extra core transformation $X_{1}$ in \eqref{eq:first_bulge},  which is in fact the cause of 
the bulge.  $X_{1}$ is the \emph{misfit}.  Instead of chasing the bulge, we will chase the 
misfit through the matrix until it disappears
at the bottom.  We therefore call this a \emph{core chasing} algorithm.  

Proceeding from \eqref{eq:first_bulge}, the next step is to do a turnover $Q_{1}Q_{2}X_{1} = U_{2}\hat{Q}_{1}\hat{Q}_{2}$.
Core transformations $\hat{Q}_{1}$ and $\hat{Q}_{2}$ become the new $Q_{1}$ and $Q_{2}$.  Pictorially
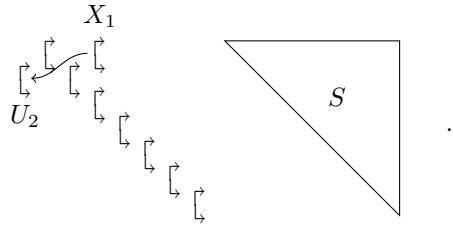
\begin{center}
  \begin{tikzpicture}[scale=1.66,y=-1cm]
    \foreach \j in {0.0,0.2,0.4,...,1.2} {
      \tikzrotation{\j}{\j}
    }  
    \uppertriangular{1.4}{0.0}
   \node[above] at (2.3,.6) {$S$};
    \node[above] at (3.2,.8) {.};      
    
    \pgfmathsetmacro{\xuone}{-.2}
    \pgfmathsetmacro{\xutwo}{0.4}
    \pgfmathsetmacro{\xuthree}{3.9}
    
    \tikzrotation[black][][][$X_{1}$]{\xutwo}{0}
    \tikzrotation[black][][${U}_{2}$]{\xuone}{0.2}

    \turnoverrl{\xutwo}{0.0}{\xuone}{0.2}
    

  \end{tikzpicture}
\end{center}
Next we do a similarity transformation, multiplying by $U_{2}^{*}$ on the left and $U_{2}$ on the right.  
This has the effect of moving $U_{2}$ from the left side to the right side of the matrix.  We can also pass
$U_{2}$ through $S$ to obtain
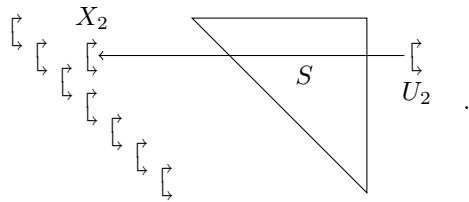
\begin{center}
  \begin{tikzpicture}[scale=1.66,y=-1cm]
    \foreach \j in {0.0,0.2,0.4,...,1.2} { \tikzrotation{\j}{\j} }
    \uppertriangular{1.4}{0.0} 
    \node[above] at (2.3,.6) {$S$}; 
    \node[above] at (3.6,.8) {.};

    \pgfmathsetmacro{\xuone}{-.3} 
    \pgfmathsetmacro{\xutwo}{0.6}
    \pgfmathsetmacro{\xuthree}{3.2} 
    \pgfmathsetmacro{\xufour}{3.9}
    \pgfmathsetmacro{\xufive}{5.9}

    \tikzrotation{\xutwo}{0.2} 
    \node at (\xutwo,0.0) [align=center] {$X_{2}$};
    \shiftthroughrl{\xuthree}{\xutwo}{0.2} 
    \tikzrotation[black][][$U_{2}$]{3.2}{0.2}

  \end{tikzpicture}
\end{center}
Now we are in the same position as we were at \eqref{eq:first_bulge}, except that the misfit
has moved downward one position.  The process continues as before:  
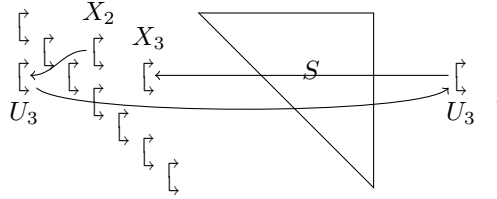
\begin{center}
  \begin{tikzpicture}[scale=1.66,y=-1cm]
    \foreach \j in {0.0,0.2,0.4,...,1.2} {
      \tikzrotation{\j}{\j}
    }  
    \uppertriangular{1.4}{0.0}
    \node[above] at (2.3,.6) {$S$}; 
    \node[above] at (3.8,.8) {.};      
    
    \pgfmathsetmacro{\xuone}{0.0}
    \pgfmathsetmacro{\xutwo}{0.6}
    \pgfmathsetmacro{\xuthree}{3.5}
    \pgfmathsetmacro{\xufour}{1.0}
    
    \tikzrotation[black][][][$X_{2}$]{\xutwo}{0.2}
    \tikzrotation[black][][][$X_{3}$]{\xufour}{0.4}
    \tikzrotation[black][][${U}_{3}$]{\xuone}{0.4}

    \turnoverrl{\xutwo}{0.2}{\xuone}{0.4}
    
    \tikzrotation[black][][$U_{3}$]{\xuthree}{0.4}
    \transferbulgelr{\xuone}{\xuthree}{0.4}
    \shiftthroughrl{\xuthree}{\xufour}{0.4} 

  \end{tikzpicture}
\end{center}

After $n-1$ such steps we arrive at
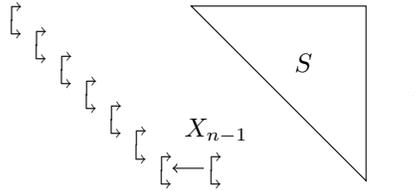
\begin{center}
  \begin{tikzpicture}[scale=1.66,y=-1cm]
    \foreach \j in {0.0,0.2,0.4,...,1.2} {
      \tikzrotation{\j}{\j}
    }  
    \uppertriangular{1.4}{0.0}
    \node[above] at (2.3,.6) {$S$}; 
    \node[above] at (3.2,.8) {.};      

    \pgfmathsetmacro{\xuone}{1.6}
    \tikzrotation[black][][][$X_{n-1}$]{\xuone}{1.2}
    \path[->,out=180,in=0] (\xuone-0.1,1.3) edge (\xuone-0.35,1.3);
  \end{tikzpicture}
\end{center}
We can now fuse $X_{n-1}$ with $Q_{n-1}$, completing the iteration.  

Exploiting the representation of the factors $W$ and $V$ we get as a total cost for passing
a core transformation through $VW^{-1} =QS$ five turnovers.  The corresponding
cost for companion QR \cite{AuMaVaWa15} is three turnovers, so we expect the companion QZ
code to be slower than the companion QR code by a factor of $5/3$.  During a QR or QZ
iteration the misfit gets passed through the matrix about $n$ times, so the total number
of turnovers is $5n$ for a QZ step and $3n$ for a QR step.  Either way the flop count is
$O(n)$.  Reckoning $O(n)$ total iterations, we get a flop count of $O(n^{2})$ for both
companion $QR$ and $QZ$, with companion $QZ$ expected to be slower by a factor of $5/3$.

In Figure~\ref{fig:runtime} we show execution times for our companion QR and QZ
codes on polynomials of degree from 4 up to about 16000.  We also make a 
comparison with
the code from Boito, Eidelman, and Gemignani (BEGQZ) \cite{BoEiGe14} and the
LAPACK QR and QZ codes ZHSEQR and ZHGEQZ.  Straight lines indicating $O(n^{2})$
and $O(n^{3})$ performance are included for comparison purposes.  Our codes are
the fastest.  The lower plot in the figure corroborates our expectation that
companion QZ will be slower than companion QR by a factor of about $5/3$.  For
polynomials of low degree the LAPACK QR code DHSEQR is roughly as fast as our
codes, but from degree 100 or so we are much faster than all other methods.  The
execution time curves for our companion QR and QZ codes are about parallel to
the $O(n^{2})$ line, indicating $O(n^{2})$ execution time.  The same is true of
the BEGQZ method.

Table~\ref{tab:runtime} shows the execution times for a few selected high degrees.

For this experiment we used a single core of an Intel\textregistered{}
Xeon\textregistered{} CPU E5-2697 v3 running at 2.60GHz with 35 MB shared cache and
128 GB RAM. We used the GNU Compiler Collection gcc version 4.8.5 on an Ubuntu
14.04.1. For the comparison we used LAPACK version 3.7.1.


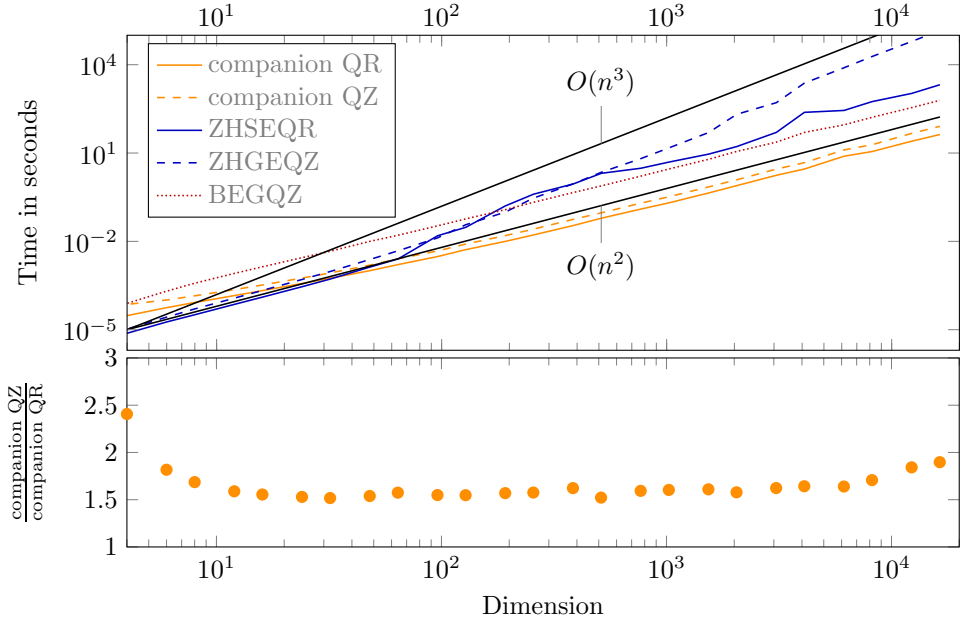
\begin{figure}
  \centering
  \begin{tikzpicture}
    \begin{loglogaxis}[
      name=fig41,
      scale only axis,
      xmin=4, xmax=20000,
      ymin=2e-6, ymax=100000,
      xticklabel pos=right,
      ylabel={Time in seconds},
      every axis y label/.style= {at={(-0.12,0.5)},rotate=90},
      width=0.85\textwidth,
      height=0.20\textheight,
      axis on top,
      legend style = {
        fill opacity = 0.5,      
        draw opacity = 1,
        text opacity = 1,
        cells = {anchor = west},
        legend pos = north west,
        at = {(0.025,0.975)}
      },
      ]

      \addplot[SPECorange, semithick]
      table [x=size,y=runtime]{dat/rt_qr.dat};
  
      \addplot[SPECorange, semithick, dashed]
      table [x=size,y=runtime]{dat/rt_qz.dat};  

      \addplot[SPECblue, semithick]
      table [x=size,y=runtime]{dat/rt_laqr.dat};  

      \addplot[SPECblue, semithick, dashed]
      table [x=size,y=runtime]{dat/rt_laqz.dat};  

      \addplot[SPECred, semithick, densely dotted]
      table [x=size,y=runtime]{dat/rt_begqz.dat};  
      
      \addplot[black, semithick] coordinates{
        (4,1e-5)
        (8,4e-5)
        (16,1.6e-4)
        (32,6.4e-4)
        (64,2.56e-3)
        (128,1.024e-2)
        (256,4.096e-2)
        (512,1.6384e-1)
        (16384,1.6777216e+2)
      };
      \node[coordinate,pin={[pin distance=.50cm]below:{$O(n^{2})$}}]
      at (axis cs:  512 , 1.6584e-1 ){};
      
      \addplot[black, semithick] coordinates{
        (4,1e-5)
        (8,8e-5)
        (16,6.4e-4)
        (32,5.12e-3)
        (64,4.096e-2)
        (128,3.2768e-1)
        (256,2.62144e-0)
        (512,2.097152e+1)
        (16384,6.8719476736e+5)
      };
      \node[coordinate,pin={[pin distance=.50cm]above:{$O(n^{3})$}}]
      at (axis cs:  512 , 2.097152e+1 ){};

      \legend{companion QR, companion QZ, ZHSEQR, ZHGEQZ, BEGQZ};
    \end{loglogaxis}

    \begin{axis}[
      xmode=log,
      at = {($(fig41.south)+(0,-3pt)$)},
      anchor = north,
      scale only axis,
      xmin=4, xmax=20000,
      ymin=1, ymax=3,
      xlabel={Dimension},
      ylabel={$\frac{\text{companion QZ}}{\text{companion QR}}$},
      every axis y label/.style= {at={(-0.12,0.5)},rotate=90},
      width=0.85\textwidth,
      height=0.12\textheight,
      axis on top,
      ]
      
      \addplot[SPECorange, semithick,only marks]%
      coordinates { ( 3 , 5.7682E-05 / 1.7762E-05 ) %
        ( 4 , 7.2266E-05 / 3.0029E-05 ) %
        ( 6 , 1.0183E-04 / 5.6044E-05 ) %
        ( 8 , 1.4160E-04 / 8.3984E-05 ) %
        ( 12 , 2.2930E-04 / 1.4432E-04 ) %
        ( 16 , 3.3105E-04 / 2.1289E-04 ) %
        ( 24 , 5.6745E-04 / 3.7097E-04 ) %
        ( 32 , 8.4766E-04 / 5.5859E-04 ) %
        ( 48 , 1.5572E-03 / 1.0117E-03 ) %
        ( 64 , 2.5586E-03 / 1.6250E-03 ) %
        ( 96 , 4.7294E-03 / 3.0529E-03 ) %
        ( 128 , 8.1016E-03 / 5.2344E-03 ) %
        ( 192 , 1.5659E-02 / 9.9765E-03 ) %
        ( 256 , 2.6047E-02 / 1.6531E-02 ) %
        ( 384 , 5.5786E-02 / 3.4381E-02 ) %
        ( 512 , 9.2938E-02 / 6.1063E-02 ) %
        ( 768 , 1.9890E-01 / 1.2481E-01 ) %
        ( 1024 , 3.2663E-01 / 2.0369E-01 ) %
        ( 1536 , 7.0720E-01 / 4.3930E-01 ) %
        ( 2048 , 1.2331E+00 / 7.8138E-01 ) %
        ( 3072 , 2.8442E+00 / 1.7522E+00 ) %
        ( 4096 , 4.7130E+00 / 2.8695E+00 ) %
        ( 6144 , 1.2774E+01 / 7.7915E+00 ) %
        ( 8192 , 1.9224E+01 / 1.1260E+01 ) %
        ( 12288 , 4.6581E+01 / 2.5283E+01 ) %
        ( 16384 , 8.0829E+01 / 4.2606E+01 ) %
      };
    \end{axis}

  \end{tikzpicture}

\caption{Execution times for several methods on polynomials of degree from 4 up to 16384.  $O(n^{2})$ and 
$O(n^{3})$ lines are included for comparison.}
\label{fig:runtime}
\end{figure}

\begin{table}
\caption{Execution times in seconds for selected high degrees}
\label{tab:runtime}
\begin{center}
  \begin{tabular}{l rrr}
\toprule
degree              & 3072 & 6144 & 12288  \\ \midrule
comp.\ QR & 2 & 8 & 25 \\
comp.\ QZ & 3 & 13 & 47 \\
BEGQZ               & 23 & 91 & 350 \\
    ZHSEQR         & 50 & 280 & 1062 \\
ZHGEQZ         & 514 & 7885 & 61856\\ 
\bottomrule
\end{tabular}
\end{center}
\end{table}

\section{Backward Error Analysis}
\label{sec:stab}

The norm symbol $\norm{\cdot}$ will denote the $2$-norm, i.e.\ the Euclidean 
norm for vectors and the spectral norm for matrices.  These are the norms that we use in our 
backward error analysis.  In our numerical tests we use the Frobenius 
matrix norm $\norm{\cdot}_{F}$ instead of the spectral norm because it is easier to compute.  
In fact the choice of norms is not important to the analysis; we could use other common norms such
as $\norm{\cdot}_{1}$ or $\norm{\cdot}_{\infty}$ with no change in the results.  

In addition we use the following conventions:  $\doteq$ denotes an
 equality where second and higher order terms are dropped,
 $\lesssim$ stands for less than or equal up to a modest
 multiplicative constant typically depending on $n$ as a low-degree polynomial, $\approx$ denotes
 equal up to a modest multiplicative constant.  The symbol $u$ denotes the \emph{unit roundoff}, which
 is about $10^{-16}$ for IEEE binary64 arithmetic.  

Van Dooren and Dewilde \cite{n643} were the first to investigate the backward stability of polynomial
root finding via companion matrices and pencils. 
Edelman and Murakami \cite{m460} revisited this analysis, focusing on scalar polynomials.
 J{\'o}nsson and Vavasis \cite{JoVa04} presented a clear summary of these results.
 
 There are two important measures of backward accuracy when dealing with companions: the
 backward error (i) on the companion matrix or pencil and (ii) on the 
 coefficients of the original polynomial. Let $a = \left[\begin{array}{ccc} a_{0} & \cdots & a_{n}\end{array}\right]^{T}$, 
 the coefficient vector of $p$.  Edelman and Murakami \cite{m460} showed that
 pushing an unstructured error further back from the pencil (or matrix) to the polynomial coefficients
 introduces an additional factor $\norm{a}$ in the backward error.  We will show that we
 can do better since the backward error produced by our companion QR method is highly structured.

\begin{figure}
  \centering
  \begin{tikzpicture}%
  \begin{loglogaxis}[%
    clip mode=individual,
    axis background/.style={fill=white},
    name=plot51a,%
    scale only axis,%
    xmin=1e0, xmax=1e12,%
    ymin=1e-18, ymax=1e+8,%
    xticklabel pos=left,
    xtickmin = 10,
    yticklabel pos=left,
    ylabel={$\|\delta A\|$},
    every axis y label/.style= {at={(-0.25,0.5)},rotate=90},%
    width=0.36\textwidth,%
    height=0.18\textheight,%
    axis on top,%
    ]

    \addplot[SPECred, semithick,only marks, mark=triangle]%
    table [x=normamon,y=LQR||A-Atilde||]{dat/be_0.dat};  
      
    \addplot[black, semithick] coordinates{
      (1e00,1e-14)
      (1e04,1e-10)
      (1e08,1e-06)
      (1e12,1e-02)
    };
    \node[coordinate,pin={[pin distance=.50cm]below:{$\norm{a}$}}]
    at (axis cs:  5e+8 , 5e-6 ){};
    \addplot[black, semithick] coordinates{
      (1e00,1e-14)
      (1e04,1e-06)
      (1e08,1e+02)
      (1e12,1e+10)
    };
    \node[coordinate,pin={[pin distance=.50cm]above:{$\norm{a}^{2}$}}]
    at (axis cs:  1e+7 , 1e-0 ){};
  \end{loglogaxis}

  \begin{loglogaxis}[%
    clip mode=individual,
    axis background/.style={fill=white},%
    name=plot51b,%
    at = {($(plot51a.east)+(+3pt,0)$)},%
    anchor = west,%
    scale only axis,%
    xmin=1e0, xmax=1e12,%
    ymin=1e-18, ymax=1e+8,%
    yticklabel pos=right,%
    ylabel={$\norm{a-\hat{a}}$},%
    every axis y label/.style= {at={(1.25,0.5)},rotate=-90},%
    width=0.36\textwidth,%
    height=0.18\textheight,%
    axis on top,%
    ]
    
    \addplot[SPECblue, semithick,only marks, mark=square]%
    table [x=normamon,y=LQR||a-atilde||]{dat/be_0.dat};  

    \addplot[black, semithick] coordinates{
      (1e00,1e-14)
      (1e04,1e-10)
      (1e08,1e-06)
      (1e12,1e-02)
    };
    \node[coordinate,pin={[pin distance=.50cm]below:{$\norm{a}$}}]
    at (axis cs:  5e+8 , 5e-6 ){};
    \addplot[black, semithick] coordinates{
      (1e00,1e-14)
      (1e04,1e-06)
      (1e08,1e+02)
      (1e12,1e+10)
    };
    \node[coordinate,pin={[pin distance=.50cm]above:{$\norm{a}^{2}$}}]
    at (axis cs:  1e+7 , 1e-0 ){};

  \end{loglogaxis}
  \end{tikzpicture}
  \caption{Backward error on the companion matrix (left) and the coefficient vector (right) as a function of 
    $\norm{a}$ when roots are computed by unstructured LAPACK code.}\label{fig:1} 
\end{figure}
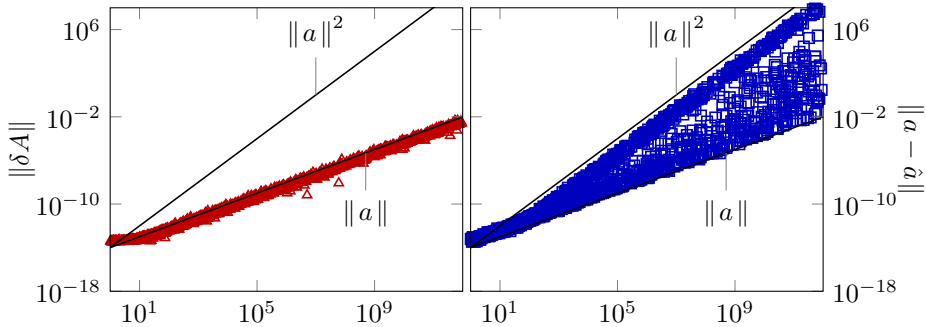
\begin{figure}
  \centering
  \begin{tikzpicture}%

      

  
  \begin{loglogaxis}[%
    clip mode=individual,
    axis background/.style={fill=white},%
    name=plot52b,%
    anchor = west,%
    scale only axis,%
    xmin=1e0, xmax=1e12,%
    ymin=1e-18, ymax=1e+8,%
    yticklabel pos=left,%
    ylabel={$\norm{a-\hat{a}}$},%
    every axis y label/.style= {at={(-0.25,0.5)},rotate=90},%
    width=0.36\textwidth,%
    height=0.18\textheight,%
    axis on top,%
    ]

    \addplot[SPECblue, semithick,only marks, mark=square]%
    table [x=normamon,y=bLQR||a-atilde||]{dat/be_0.dat};  

    \addplot[black, semithick] coordinates{
      (1e00,1e-14)
      (1e04,1e-10)
      (1e08,1e-06)
      (1e12,1e-02)
    };
    \node[coordinate,pin={[pin distance=.50cm]below:{$\norm{a}$}}]
    at (axis cs:  5e+8 , 5e-6 ){};
    \addplot[black, semithick] coordinates{
      (1e00,1e-14)
      (1e04,1e-06)
      (1e08,1e+02)
      (1e12,1e+10)
    };
    \node[coordinate,pin={[pin distance=.50cm]above:{$\norm{a}^{2}$}}]
    at (axis cs:  1e+7 , 1e-0 ){};

  \end{loglogaxis}
  \end{tikzpicture}
\caption{Same experiment as in Figure~\ref{fig:1}, except that the matrix is balanced before the eigenvalue 
computation.  Only the backward error on the polynomial coefficient vector is shown.}\label{fig:2} 
\end{figure}
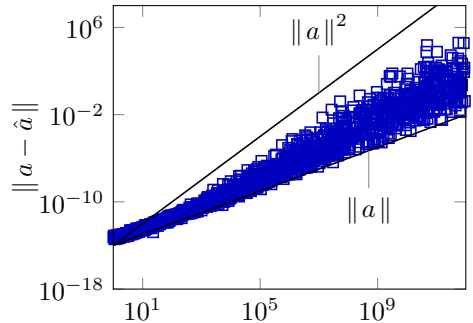

To illustrate what happens in the unstructured case 
we show in Figure~\ref{fig:1} the absolute backward error as a function of $\norm{a}$
when the companion matrix eigenvalue problem is solved  by the 
QR algorithm from LAPACK.  Twelve hundred random polynomials with varying norm 
between $1$ and $10^{12}$ 
are produced by a method described below.  For each sample we plot 
the backward error against $\norm{a}$ (a single point).  Black lines with slopes corresponding to $\norm{a}$ and 
$\norm{a}^{2}$ performance are also provided.  

The graph on the left is the backward error on the companion matrix $A$.  
We see that this grows linearly as a function of $\norm{a}$.  This is consistent with the 
backward stability of the QR algorithm, which guarantees that the computed eigenvalues are the
exact eigenvalues of a slightly perturbed matrix $A + \delta A$ with 
$\norm{\delta A} \lesssim u\norm{A} \approx u\norm{a}$.

The graph on the right shows the backward error on the coefficients of the polynomial as a function
of $\norm{a}$.  Note that the growth is quadratic in $\norm{a}$ in the worst cases, 
consistent with the analysis of Edelman and Murakami \cite{m460}.

In this and all subsequent experiments we used polynomials of varying norm such
that the coefficients within each polynomial have widely varying magnitude.  
Specifically, we produced polynomials with increasing coefficient norm, parametrized by 
an integer $\rho = 1, \dotsc, 12$.
For each $\rho$ we ran 100 experiments. 
We chose polynomials of degree 50, but the results for degree 6,
20, and 200 were very similar. For each of the 51 coefficients of each
polynomial we choose three random numbers $\nu$, $\mu$, and $\eta$, uniformly
distributed in $[0,1]$.  The coefficient is a complex number with 
argument $2\pi\nu$ and  absolute value
$(2\mu-1) 10^{\rho(2\eta - 1)}$. 
Polynomials of similar type have been used by De Ter{\'a}n, Dopico, and P\'erez in \cite{DeDoPe16}.

In this and other experiments for which monic polynomials were needed, we made them monic by 
dividing through by the leading coefficient $a_{n}$.

Balancing \cite{ParRei69} is often touted as an important step in solving the companion eigenvalue
problem.  In Figure~\ref{fig:2} we have repeated the same experiment as in Figure~\ref{fig:1}, except that 
a balancing step is added.  We see that balancing reduces the backward error on the polynomial coefficients.
Later on we will show that our companion QR algorithm (with no balancing) has a 
much better backward error.\footnote{However, a smaller backward error does not necessarily imply more 
accurate roots.  It can happen that the balancing operation improves the condition number of the eigenvalue 
problem enough to more than offset the disadvantage in backward error.  This is a matter for further study.}  
As a step in this direction we begin with some basic facts about the turnover.  

\subsection{Backward Stability of the Turnover}\label{subsec:newturn}

Anyone who has ever tried to program the turnover operation has discovered that there are many ways to do
it, and some are much more accurate than others.   As we have noticed recently, 
even among the ``good'' implementations, some are better than others.

First of all, as has been known for many years, it is crucial to maintain the unitary property of the 
core transformations.  Each time a new core transformation is produced by computation of a $c$ and an $s$,
the unitary property $\absval{c}^{2} + \absval{s}^{2} = 1$ must be enforced by a rescaling operation.   Assuming that $c$ 
and $s$ have been produced by valid formulas, the magnitude of the correction will be extremely tiny, on the order of the 
unit roundoff.  This seemingly trivial correction is crucial.  If it is neglected, the core transformations will gradually
drift away from unitarity over the course of many operations, and the algorithm will fail.   If it is not neglected, the 
programmer has a good chance of producing an accurate turnover.  

It is worth mentioning as an aside that each of these rescaling operations requires a square root, 
but the number whose square root is being 
taken differs from 1 by a very small amount on the order of the unit roundoff.  Therefore we can use the 
Taylor expansion $\sqrt{1+w} = 1 + .5\,w + O(w^{2})$ to compute the square root to 
full precision cheaply or, more conveniently, we can use the formula $1/\sqrt{1+w} = 1 - .5\,w + O(w^{2})$ to 
compute the reciprocal of the square root.  
Since the cost of square root operations contributes significantly to the cost of doing a turnover, 
and both the companion QR and companion QZ algorithms are dominated by turnovers,  this shortcut has a 
significant impact.  

In the data structure for our companion QR code, the core transformations are arranged in descending and ascending 
sequences
$Q = Q_{1} \cdots Q_{n-1}$,  $B = B_{1} \cdots B_{n}$,  and $C^{*} = C_{n}^{*} \cdots C_{1}^{*}$ 
(\ref{eq:factorize:upper:triangular}).  Obviously the last of these is equivalent to the descending sequence
$C = C_{1} \cdots C_{n}$.   In the course of the algorithm these sequences are modified repeatedly by 
turnover operations and (in $Q$ only) occasional fusions.  The following theorem shows that the backward 
errors associated with these modifications are tiny, on the order of the unit roundoff $u$.

\begin{theorem}\label{thm:stableunitary}
Suppose that in the course of the companion QR algorithm the descending sequence $G = G_{1} \cdots G_{n}$
is transformed to $\hat{G}$ by multiplication on the right by $U$ and on the left by $V^{*}$, where each of $U$ and $V$ 
is a product of core transformations.   Then, in floating-point arithmetic
\begin{equation*}
\hat{G} = V^{*}(G + \delta G)U,
\end{equation*}
where  the backward error $\delta G$ satisfies $\norm{\delta G} \lesssim u$.  
\end{theorem}

Here we can view $U$ as the product of all core transformations that are absorbed by $G$ on the right via
turnovers (or fusions), and $V$ is the product of all the core transformations that were ejected from $G$ on 
the left.  This is the right view if we are passing core transformations from right to left, but one can equally
well think of passing core transformations from left to right.  Then $V^{*}$ is the product of all the transformations
that are absorbed by $G$ on the left, and $U^{*}$ is the product of the core transformations that are ejected on the 
right.  Either way Theorem~\ref{thm:stableunitary} is valid.

\begin{proof}
Each fusion is just a matrix multiplication operation, and this is backward stable.   Each turnover 
begins with a matrix multiplication operation that forms an essentially $3 \times 3$ unitary matrix.  That matrix is
then refactored by a process that is in fact a $QR$ decomposition operation (on a unitary matrix!).  These operations
are all normwise backward stable \cite{Hig02}.   Combining the small backward errors, and considering that all 
transformations are unitary and therefore do not amplify the errors, we conclude that the combined backward error
$\delta G$ satisfies $\norm{\delta G} \lesssim u$. 
\end{proof}

With Theorem~\ref{thm:stableunitary} in hand we already have enough to build a satisfactory error analysis,
but we can make the analysis stronger (and simpler) 
by taking into account  a certain quantity that is conserved by 
the turnover operation.  In $G= G_{1} \cdots G_{n}$ each core transformation $G_{i}$ has active part
\begin{equation*}
\left[\begin{array}{cc} c_{i} & -s_{i} \\ s_{i} & \phantom{-}\overline{c_{i}}\end{array}\right],
\end{equation*}
where $s_{i}$ is real.\footnote{It is always possible to make the $s_{i}$ real, and that is what we have done in our
codes.  It allows for some simplification, but it is not crucial to the theory.}  The interesting quantity that is preserved is the product $s_{1} \cdots s_{n}$, as
the following theorem shows.  Exact arithmetic is assumed.

\begin{theorem}\label{thm:sprodcons}
Suppose $G = G_{1} \cdots G_{n}$ is transformed to $\hat{G} = \hat{G}_{1} \cdots \hat{G}_{n}$ by turnovers
only (no fusions).   
Then 
\begin{equation*}
\hat{s}_{1}\hat{s}_{2} \cdots \hat{s}_{n} = s_{1} s_{2}\cdots s_{n}.
\end{equation*}
\end{theorem}

\begin{proof}
It suffices to consider a single turnover.  For the sake of argument we will think of passing a core 
transformation from right to left through $G$, but we could equally well go from left to right.  
Each turnover alters exactly two adjacent core transformations in the descending sequence.  
Before the turnover we have three adjacent cores 
\begin{equation}\label{eq:corebefore}
G_{j}G_{j+1}M_{j} =
  \begin{bmatrix}
   c_{j} & -s_{j}&\\
   s_{j} & \phantom{-}\overline{c}_{j}&\\
   &&1
 \end{bmatrix}
 \begin{bmatrix}
   1&&\\
   &c_{j+1} & -s_{j+1}\\
   &s_{j+1} & \phantom{-}\overline{c}_{j+1}
 \end{bmatrix}
 \begin{bmatrix}
   \gamma_{j} & -\sigma_{j}&\\
   \sigma_{j} & \phantom{-}\overline{\gamma}_{j}&\\
   &&1
 \end{bmatrix}.
 \end{equation}
$G_{j}$ and $G_{j+1}$ belong to the descending sequence, and $M_{j}$ is the misfit coming in from 
the right.  We form the  $3 \times 3$ product, which we then factor into three new core transformations
\begin{equation}\label{eq:coreafter}
M_{j+1}\hat{G}_{j}\hat{G}_{j+1} =
 \begin{bmatrix}
   1&&\\
   &\gamma_{j+1} & -\sigma_{j+1}\\
   &\sigma_{j+1} & \phantom{-}\overline{\gamma}_{j+1}
 \end{bmatrix}
   \begin{bmatrix}
   \hat{c}_{j} & -\hat{s}_{j}&\\
   \hat{s}_{j} & \phantom{-}\overline{\hat{c}}_{j}&\\
   &&1
 \end{bmatrix}
 \begin{bmatrix}
   1&&\\
   &\hat{c}_{j+1} & -\hat{s}_{j+1}\\
   &\hat{s}_{j+1} & \phantom{-}\overline{\hat{c}}_{j+1}
 \end{bmatrix}.
\end{equation}
$M_{j+1}$ will be the new misfit, which is going to be ejected to the left, while $\hat{G}_{j}$ and 
$\hat{G}_{j+1}$ will replace $G_{j}$ and $G_{j+1}$ in the descending sequence.  

It is not hard to work out the products (\ref{eq:corebefore}) and (\ref{eq:coreafter}) explicitly, 
but for our present purposes we just need to compute the  $(1,3)$ entry of each.  
These are easily seen to be $s_{j}s_{j+1}$ 
and $\hat{s}_{j}\hat{s}_{j+1}$, respectively.  Since (\ref{eq:corebefore}) and (\ref{eq:coreafter}) are equal, we conclude 
that $\hat{s}_{j}\hat{s}_{j+1} = s_{j}s_{j+1}$.  
This proves the theorem.    
\end{proof}

We remark that Theorem~\ref{thm:sprodcons} is not applicable to the 
$Q$ sequence, as $Q$ is subjected to fusions. Moreover, $Q$ is the sequence in which we
 search for deflations: a zero $s_i$ is a lucky event signaling a deflation. The
$B$ and $C$ sequences satisfy the conditions of Theorem~\ref{thm:sprodcons} ensuring
thereby that all turnovers are well-defined implying correctness of the algorithm \cite{AuMaVaWa15}.

Next we consider how Theorem~\ref{thm:sprodcons} holds up in floating-point arithmetic.
It turns out that this depends upon how the turnover is implemented.  Using notation from the proof of the 
theorem, suppose we have computed $\hat{s}_{j}$ by some means.  
Then  we can compute $\hat{s}_{j+1}$ by 
\begin{equation}\label{eq:goodsr}
\hat{s}_{j+1} = \frac{s_{j}s_{j+1}}{\hat{s}_{j}}.
\end{equation}
If the turnover uses this formula, the product will be preserved to high relative accuracy.  The multiplication 
and the division will each have a tiny relative error not exceeding $u$, and the computed $\hat{s}_{j}$ and $\hat{s}_{j+1}$
will satisfy $\hat{s}_{j}\hat{s}_{j+1} = s_{j}s_{j+1}(1 + \delta)$, where $\absval{\delta} \lesssim u$.  This proves the following 
theorem.

\begin{theorem}\label{thm:sprodfloat}
Suppose $G = G_{1} \cdots G_{n}$ is transformed to $\hat{G} = \hat{G}_{1} \cdots \hat{G}_{n}$ by turnovers
only (no fusions) using floating point arithmetic.  Suppose that the turnover uses the formula (\ref{eq:goodsr}).
Then 
\begin{equation*}
\hat{s}_{1}\hat{s}_{2} \cdots \hat{s}_{n} = s_{1} s_{2}\cdots s_{n}(1 + \delta),
\end{equation*}
where $\absval{\delta} \lesssim u$.
\end{theorem}

In words, the product $s_{1} \cdots s_{n}$ is preserved to high relative accuracy.  
The results presented in this paper were obtained using a (new!) turnover that uses the formula
(\ref{eq:goodsr}), so we will be able to use Theorem~\ref{thm:sprodfloat}.

The good results in our publication \cite{AuMaVaWa15} were obtained using a turnover (the \emph{old}
turnover) that did not
use (\ref{eq:goodsr}) and violated Theorem~\ref{thm:sprodfloat}.   That turnover preserves $s_{1} \cdots s_{n}$
to high absolute accuracy but not to high relative accuracy.   A satisfactory backward error analysis is still possible, 
and we did one.  We have preserved that analysis in the technical report \cite{TW683}.  We were about to submit 
it for publication when we realized that a small change in the turnover yields more accurate results and a stronger
and simpler backward error analysis.  That is what we are presenting here.

\subsection{Backward error on the companion matrix}

We consider the companion QR algorithm (with the improved turnover) first, leaving companion QZ 
for later.  We start with the backward error on the companion matrix.
We will fix the error in the analysis of \cite{AuMaVaWa15} and make other 
substantial improvements.    We will take $p$ to be monic with coefficient vector 
$a = \left[\begin{array}{cccccc} a_{0} & \cdots & a_{n-1} & 1\end{array}\right]^{T}$.  The companion
matrix is 
\begin{align}
  A = \begin{bmatrix}
    &        &   & -a_{0} \\
    1 &        &   & -a_{1} \\
    & \ddots &   & \vdots \\
    &        & 1 & -a_{n-1}
  \end{bmatrix}.
\label{eq:moniccompanionmatrix}
\end{align} 
Clearly $1 \leq \norm{a} \approx \norm{A}$.

When we run the companion QR algorithm on $A$, 
we transform it to $U^{*}AU$, where $U$ is the product of all of the core transformations that took part in similarity transformations.  At the same time $Q$ and $R$ are transformed to 
$U^{*}QX$ and $X^{*}RU$, where $X$ is the product of all core transformations that were ejected from $R$ and 
absorbed by $Q$.  (Similarly we can view $U$ as the product of all core transformations that were ejected from 
$Q$ and absorbed by $R$.)

In floating-point arithmetic we have 
$$\hat{A} = U^{*}(A + \delta A)U.$$
where $\delta A$ is the backward error.   The roots that we compute are
exactly the eigenvalues of $\hat{A}$ and of $A + \delta A$.  We would like to get a bound on $\norm{\delta A}$.

We begin by looking at the backward error on $R$, and to this end we consider first the larger matrix 
$\underline{R}$  \eqref{eq:bigr}, which we can write in the factored form \eqref{eq:factorize:upper:triangular}.  

$$\underline{\hat{R}} = \underline{X}^{*}(\underline{R} + \underline{\delta R})\underline{U},$$
where $\underline{U} = \diagg{U,1}$ and $\underline{X} = \diagg{X,1}$.  
Consider the unitary and rank-one parts separately:   $\underline{R} = \underline{R}_{u} + \underline{R}_{o}$, where
$\underline{R}_{u} = C^{*}B$ and $\underline{R}_{o} = \alpha\, C^{*}e_{1}\underline{y}^{T} = 
\alpha\,\underline{x}\,\underline{y}^{T}$.  Here we have introduced a new symbol $\underline{x} = C^{*}e_{1} = 
\alpha^{-1}\underline{z}$.  We note that $\norm{\underline{x}} = \norm{\underline{y}} = 1$, and 
$\absval{\alpha} = \norm{\underline{z}} = \norm{a}$ in this (monic) case. 
We will determine backward errors
associated with these two parts:  $\underline{\delta R} = \underline{\delta R}_{u} + \underline{\delta R}_{o}$.
 
Since $\underline{R}_{u} = C^{*}B$, we have 
$\underline{R}_{u} + \underline{\delta R}_{u} = (C + \delta C)^{*}(B + \delta B)$, where 
$\norm{\delta B} \lesssim u$ and $\norm{\delta C} \lesssim u$ by Theorem~\ref{thm:stableunitary}.
We deduce that $\underline{\delta R}_{u} \doteq 
\delta C^{*}B + C^{*}\delta B$, and $\norm{\underline{\delta R}_{u}} \lesssim u$.

For the rank-one part recall from \eqref{eq:yrecover} that
$$\alpha\underline{y}^{T} = - \rho^{-1}e_{n+1}^{T}\underline{R}_{u},\quad \mbox{where} \quad 
\rho = e_{n+1}^{T}C^{*}e_{1}.$$
It is not hard to show that the quantity $\rho$ remains invariant under QR iterations in exact arithmetic.
Our mistake in the backward error analysis in \cite{AuMaVaWa15} was to treat it as a constant, when in 
fact we should have taken the error into account.  The simplest way to do this is to note that 
$$\hat{\rho} = \rho + \delta \rho = e_{n+1}^{T}(C + \delta C)^{*}e_{1},$$
so $\delta \rho = e_{n+1}^{T}\delta C^{*}e_{1}$, and $\absval{\delta \rho} \lesssim u$.  This computation, 
which shows that $\delta \rho$ is small in an absolute sense but not necessarily small relative to $\rho$, 
is valid for both the old and the new turnover.   

With our new turnover we can get a better result:  Let  
\begin{displaymath}
\left[\begin{array}{cc} c_{i} & -s_{i} \\ s_{i} & \phantom{-}\overline{c}_{i}\end{array}\right] \quad i=1,\,\ldots,\,n
\end{displaymath}
denote the active parts of the core transformations in the descending sequence $C = C_{1}\cdots C_{n}$.  
A straightforward computation shows that $\rho = e_{n+1}C^{*}e_{1} = (-1)^{n}s_{1} \cdots s_{n}$.  Therefore
we can invoke Theorem~\ref{thm:sprodfloat} to deduce that the error in $\rho$ is tiny relative to $\rho$:
\begin{equation}\label{eq:rhobackerr}
\hat{\rho} = \rho(1 + \delta\rho_{r}) \quad\mbox{where } \absval{\delta\rho_{r}} \lesssim u.
\end{equation}
This result enables a stronger and easier error analysis.  

Another straightforward computation shows that  $-\rho^{-1} = \alpha$. Indeed, 
$$\rho = e_{n+1}^{T}C^{*}e_{1} = 
e_{n+1}^{T}\underline{x} = x_{n+1} = \alpha^{-1}z_{n+1} = -\alpha^{-1},$$
with $\underline{z}$ defined in \eqref{eq:uxform}.  
Thus $\underline{y}^{T} = e_{n+1}^{T}\underline{R}_{u}$ and
$$\underline{R}_{o} = \alpha\,\underline{x}\,\underline{y}^{T} = \alpha\,C^{*}e_{1}e_{n+1}^{T}\underline{R}_{u}.$$

The backward error in 
$\underline{R}_{o}$ is given by 
\begin{equation*}
\underline{R}_{o} + \underline{\delta R}_{o}  = 
\alpha(1 + \delta\rho_{r})^{-1}(C + \delta C)^{*}e_{1}e_{n+1}^{T}(\underline{R}_{u} + \underline{\delta R}_{u})
\end{equation*}
Making the approximation $(1 + \delta\rho_{r})^{-1} \doteq 1 - \delta\rho_{r}$, we get
\begin{equation}
\underline{R}_{o} + \underline{\delta R}_{o}  = 
\alpha(\underline{x} + \underline{\delta x})(\underline{y} + \underline{\delta y})^{T}, \label{eq:bigrpert}
\end{equation}
where $\underline{\delta x} \doteq (\delta C^{*} - \delta\rho_{r}C^{*})e_{1}$ and 
$\underline{\delta y}^{T} = e_{n+1}^{T}\underline{\delta R}_{u}$.  Note that  $\norm{\underline{\delta x}} \lesssim u$ and $\norm{\underline{\delta y}} \lesssim u$.  

Recall that $\absval{\alpha} = \norm{a} = \norm{\underline{z}} = \norm{\underline{R}_{o}} \approx \norm{R} = \norm{A}$,  
and the approximation is excellent when $\norm{A}$ is large. 
Thus we will use the approximation $\absval{\alpha} = \norm{a} \approx \norm{A}$
without further comment.  

\begin{theorem}\label{thm:backstabqr1}
If the companion eigenvalue problem is solved by our
companion QR algorithm, then 
\begin{enumerate}
\item[(a)] $\hat{A} = U^{*}(A + \delta A)U$, where $\norm{\delta A} \lesssim u \norm{A}$.
\item[(b)] $\hat{Q} = U^{*}(Q + \delta Q)X$, where $\norm{\delta Q} \lesssim u$.
\item[(c)] $\hat{R} = X^{*}(R + \delta R)U$, where $\norm{\delta R} \lesssim u \norm{A}$.
\item[(d)] More precisely, 
$$\delta R \doteq \delta R_{u} + \alpha(\delta x\,y^{T} + x\,\delta y^{T} ),$$ 
where 
$\norm{\delta R_{u}} \lesssim u$,  $\norm{\delta x} \lesssim u$, and $\norm{\delta y} \lesssim u$.  
\end{enumerate}
\end{theorem}

\begin{proof}  The work is mostly done; let's  start with part (d).  A first-order expansion of  (\ref{eq:bigrpert}) gives 
$\underline{\delta R}_{o} \doteq \alpha (\underline{\delta x}\underline{y}^{T} + \underline{x}\underline{\delta y}^{T})$, 
so 
\begin{equation*}
\delta R_{o} = P^{T}\underline{\delta R}_{o}P = \alpha( \delta x\,y^{T} + x\,\delta{y}^{T}),
\end{equation*}
where $\delta x = P^{T}\underline{\delta x}$ and so on.  This establishes (d).  

The factor $\alpha$ in the expression for $\delta R_{o}$ implies that $\norm{\delta R} \lesssim u\,\norm{a} \approx u\,\norm{A}$, so (c) holds.  Part (b) is true by Theorem~\ref{thm:stableunitary}. 
Part (a) follows from (b) and (c) because $\delta A \doteq \delta Q\,R + Q\,\delta R$.
\end{proof}

\begin{example}\label{examp:backstabqr1}
As a numerical test of Theorem~\ref{thm:backstabqr1} we computed the backward error on the companion
matrix as follows.  During the companion QR iterations we accumulated the transforming matrix $U$.  
We then computed $\check{A} = U\hat{A}U^{*}$.  The backward error on $A$ is $\delta{A} = \check{A} - A$.
Figure~\ref{fig:backstabqr1} shows that the backward error on $A$
grows in proportion to $\norm{A}$ as claimed.
\end{example}

\begin{figure}
  \centering
  \begin{tikzpicture}%
  \begin{loglogaxis}[%
    clip mode=individual,
    axis background/.style={fill=white},
    name=plot53a,%
    scale only axis,%
    xmin=1e0, xmax=1e12,%
    ymin=1e-18, ymax=1e+8,%
    xticklabel pos=left,
    xtickmin = 10,
    yticklabel pos=left,
    ylabel={$\|\delta A\|$},%
    every axis y label/.style= {at={(-0.25,0.5)},rotate=90},%
    width=0.36\textwidth,%
    height=0.18\textheight,%
    axis on top,%
    ]


    \addplot[SPECred, semithick,only marks, mark=triangle]%
    table [x=normamon,y=QR||A-Atilde||b]{dat/be_0.dat};  

    \addplot[black, semithick] coordinates{
      (1e00,1e-14)
      (1e04,1e-10)
      (1e08,1e-06)
      (1e12,1e-02)
    };
    \node[coordinate,pin={[pin distance=.50cm]below:{$\norm{a}$}}]
    at (axis cs:  5e+8 , 5e-6 ){};
    \addplot[black, semithick] coordinates{
      (1e00,1e-14)
      (1e04,1e-06)
      (1e08,1e+02)
      (1e12,1e+10)
    };
    \node[coordinate,pin={[pin distance=.50cm]above:{$\norm{a}^{2}$}}]
    at (axis cs:  1e+7 , 1e-0 ){};

  \end{loglogaxis}
  \end{tikzpicture}
  
  \caption{Companion QR method; norm of the backward error on the companion
    matrix $A$ plotted against $\norm{a}$}
  \label{fig:backstabqr1}
\end{figure}
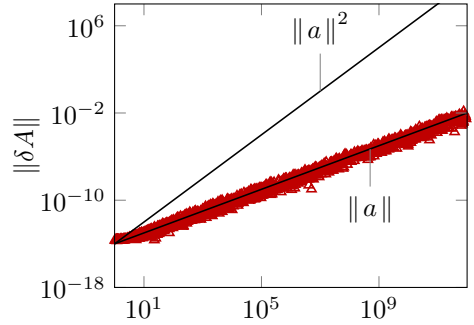

\subsection{Backward error on the polynomial coefficients}

We continue to study the backward error of the companion QR algorithm, 
leaving the companion QZ case for later.   The monic polynomial 
\begin{displaymath}
p(\lambda) = a_{0} + a_{1}\lambda +a_{2}\lambda^{2} + \cdots + a_{n-1}\lambda^{n-1} + \lambda^{n}
\end{displaymath}
is associated with the coefficient vector 
\begin{displaymath}
a = \left[\begin{array}{cccc} a_{0} & \cdots & a_{n-1} & 1\end{array}\right]^{T}\in \cplxs^{n+1}.
\end{displaymath}
 When we compute the zeros of $p$, we don't get the exact zeros, but we hope that they are the 
zeros of a ``nearby'' polynomial.  Let $\lambda_{1}$, \ldots, $\lambda_{n}$ denote the computed zeros
and 
\begin{displaymath}
\tilde{p}(\lambda)  = (\lambda - \lambda_{1}) \cdots (\lambda - \lambda_{n}).
\end{displaymath}
This is the monic polynomial that has the computed roots as its zeros.  We can compute the 
coefficients of $\tilde{p}$:
\begin{displaymath}
\tilde{p}(\lambda) = 
\tilde{a}_{0} + \tilde{a}_{1}\lambda +\tilde{a}_{2}\lambda^{2} + \cdots + \tilde{a}_{n-1}\lambda^{n-1} + \lambda^{n}
\end{displaymath}
and the corresponding coefficient vector
\begin{displaymath}
\tilde{a} = \left[\begin{array}{cccc} \tilde{a}_{0} & \cdots & \tilde{a}_{n-1} & 1\end{array}\right]^{T}\in \cplxs^{n+1}.
\end{displaymath}
This computation is done in multiple precision arithmetic using the multi
precision engine of MPSolve 3.1.5 \cite{BinLR14}.  The quantity
$\delta{a} = \tilde{a}-a$ is the backward error on the coefficients.  We would
like to show that $\norm{\delta a}$ is tiny.

Theorem~\ref{thm:backstabqr1} shows that the norm of the backward error on the companion matrix $A$
is directly proportional to the norm of the matrix.  According to 
Edelman and Murakami \cite{m460}, the backward error on the polynomial coefficients should then grow 
like the square of the norm, that is, $\norm{\delta a} \lesssim u\,\norm{a}^{2}$.
In this section we show that we can do better:  
by exploiting the structure of the problem, we can make an argument that 
shows that the backward error on the polynomial depends linearly on the 
norm:  $\norm{\delta a} \lesssim u\,\norm{a}$.   
This is an optimal result and is better than what is achieved by LAPACK's unstructured QR algorithm or 
Matlab's \texttt{roots} command.

Now let's get started on the analysis.  
In the factorization $A=QR$,  the triangular factor can be written as $R = I + (\alpha\,x - e_{n})y^{T}$, using notation
established earlier in this section.  Thus
\begin{equation}\label{eq:aqperturb}
\lambda I - A = (\lambda I - Q) - Q(\alpha\,x - e_{n})y^{T}.
\end{equation}
To make use of this equation 
we need the following known result \cite[p.\ 26]{HorJoh13}.  
\begin{lemma}\label{lem:detadj}
Let $K\in\cnxn$, and let $w$, $v \in \cn$.  Then
$$\det(K + w\,v^{T}) = \det(K) + v^{T}\adj(K)w,$$
where $\adj(K)$ denotes the adjugate matrix of $K$.
\end{lemma}


\begin{corollary}\label{cor:charpolform}
If $A = QR = Q(I + (\alpha\,x-e_{n})y^{T})$, then
$$\det(\lambda I - A) = \det(\lambda I - Q) - y^{T}\adj(\lambda I - Q)Q(\alpha\,x - e_{n}).$$  
\end{corollary}


The entries of the adjugate matrix are determinants of order $n-1$, so $\adj(\lambda I - Q)$ is a matrix
polynomial of degree $n-1$:
\begin{equation}\label{eq:adjpol}
\adj(\lambda I - Q) = \sum_{k=0}^{n} G_{k} \lambda^{k},
\end{equation}
with $\norm{G}_{k} \approx 1$ for $k=1$, \ldots, $n-1$, and $G_{n} = 0$.

The characteristic polynomial of $Q$ is 
$\det(\lambda I - Q) = \lambda^{n}-1$, which we will also write as
\begin{displaymath}
\det(\lambda I - Q) = \sum_{k=0}^{n} q_{k} \lambda^{k},
\end{displaymath}
with $q_{n} = 1$, $q_{0} = -1$, and $q_{k}=0$ otherwise.
Using Corollary~\ref{cor:charpolform} and (\ref{eq:adjpol}) we can write the characteristic polynomial of $A$ as 
\begin{equation}\label{eq:charpolform}
p(\lambda) = \det(\lambda I - A) = \sum_{k=0}^{n} \left[q_{k} - y^{T}G_{k}Q(\alpha\, x - e_{n})\right] \lambda^{k}.
\end{equation}

The roots that we actually compute are the zeros of a perturbed polynomial
\begin{displaymath}
\tilde{p}(\lambda) = \det(\lambda I - (A + \delta A)) = \sum_{k=0}^{n} (a_{k} + \delta a_{k}) \lambda^{k}.
\end{displaymath}  
Our plan now is to use (\ref{eq:charpolform}) to determine the effect of the perturbation $\delta A$
on the coefficients of the characteristic polynomial.   That is, we want bounds on $\absval{\delta a_{k}}$.
 
\begin{lemma}\label{lem:adjupert}
If $\norm{\delta Q} \lesssim u$ then
\begin{equation}\label{eq:adjupert}
\adj(\lambda I - (Q + \delta Q)) = \sum_{k=0}^{n}(G_{k} + \delta G_{k}) \lambda^{k}
\end{equation}
and 
\begin{equation}\label{eq:qpolpert}
\det(\lambda I - (Q + \delta Q)) = \sum_{k=0}^{n}(q_{k} + \delta q_{k}) \lambda^{k},
\end{equation}
with $\norm{\delta G_{k}} \lesssim u$ and $\absval{\delta q_{k}} \lesssim u$, $k=0$, \ldots, $n-1$.
\end{lemma}

\begin{proof}
For the bounds $\absval{\delta q_{k}} \lesssim u$ we rely on Edelman and Murakami \cite{m460}.
We have $\norm{\delta q} \lesssim u \,\norm{q}^{2} = 2u$.  

The adjugate and the determinant are related by the fundamental equation $B\adj(B) = \det(B) I$.  
Applying this with $B = \lambda I - Q$, we get
\begin{equation*}
(\lambda I - Q)\sum_{k=0}^{n}G_{k}\lambda^{k} = \sum_{k=0}^{n}q_{k}I\lambda^{k}.
\end{equation*}
Expanding the left-hand side and equating like powers of $\lambda$, we obtain the recurrence
\begin{equation}\label{eq:fadlev1}
G_{k} = QG_{k+1} + q_{k+1} I, \quad k = n-1,\,\ldots,\,0.
\end{equation}
This is one half of the 
Faddeev-Leverrier method\index{Faddeev-Leverrier method} \cite[p.~260]{FadFad63}, \cite[p.~87]{Gan59}.
Starting from $G_{n}=0$, and knowing the coefficients $q_{k}$,  we can use (\ref{eq:fadlev1}) to obtain 
all of the coefficients of $\adj(\lambda I - Q)$.  The recurrence holds equally well with $Q$ replaced 
by $Q + \delta Q$.  We have 
\begin{equation*}
G_{k} + \delta G_{k} = (Q + \delta Q)(G_{k+1} + \delta G_{k+1}) + (q_{k+1} + \delta q_{k+1})I,
\end{equation*}  
so
\begin{equation}\label{eq:gkpertind}
\delta G_{k} \doteq \delta Q\,G_{k+1} + Q\,\delta G_{k+1} + \delta q_{k+1}I.
\end{equation}
If $\norm{\delta G_{k+1}} \lesssim u$,  we can 
deduce from (\ref{eq:gkpertind}) that $\norm{\delta G_{k}} \lesssim u$. 
Since we have $\delta G_{n} = 0$ 
to begin with, we get by induction that $\norm{\delta G_{k}} \lesssim u$ 
for all $k$.
\hfill\end{proof}


\begin{lemma}\label{lem:rform2} 
If the companion QR algorithm is applied to the companion matrix $A = QR$, where 
$R = I + (\alpha\, x - e_{n})y^{T}$, then the backward error $\delta R$ satisfies 
\begin{displaymath}
R + \delta R = (I + \delta I) + (\alpha(x + \delta x)-e_{n})(y+\delta y)^{T},
\end{displaymath}
where $\norm{\delta I} \lesssim u$, $\norm{\delta x} \lesssim u$, and $\norm{\delta y} \lesssim u$.
\end{lemma}

\begin{proof}
From (\ref{eq:bigrpert}) it follows that
\begin{displaymath}
\underline{R} + \underline{\delta R} = \underline{R}_{u} + \underline{\delta R}_{u} 
+ \alpha\,(\underline{x} + \underline{\delta x})(\underline{y} + \underline{\delta y})^{T}.
\end{displaymath}
Projecting this down to $n \times n$ matrices we get
\begin{equation}\label{eq:rpdr}
R + \delta R = R_{u} + \delta R_{u} 
+ \alpha\,(x + \delta x)(y + \delta y)^{T},
\end{equation}
with $\norm{\delta R_{u}} \lesssim u$, 
$\norm{\delta x} \lesssim u$, and $\norm{\delta y} \lesssim u$.

Recalling the form of $\underline{R}_{u}$, we see that $R_{u} = I - e_{n}y^{T}$, so 
\begin{eqnarray}
R_{u} + \delta R_{u} & =  & I - e_{n}y^{T} + \delta R_{u} \nonumber \\
& = & I - e_{n}(y + \delta y)^{T} + e_{n}\,\delta y^{T} + \delta R_{u}  \nonumber \\
& = & I + \delta I - e_{n}(y + \delta y)^{T}, \label{eq:rupdru}
\end{eqnarray}
where $\delta I = e_{n}\delta y^{T} + \delta R_{u}$, and $\norm{\delta I} \lesssim u$.  Substituting (\ref{eq:rupdru})
into (\ref{eq:rpdr}), we get  
\begin{equation*}
R + \delta R = I + \delta I + (\alpha(x + \delta x) - e_{n})(y + \delta y)^{T},
\end{equation*}
completing the proof. \hfill\end{proof}

\begin{theorem}\label{thm:backerrmonic}
Suppose we apply the companion QR algorithm to the monic polynomial $p$ with coefficient vector $a$.  
Let $\tilde{p}$, with coefficient vector $\tilde{a} = a + \delta a$, denote the monic polynomial that has the computed 
roots as its exact zeros.  Then 
\begin{displaymath}
\norm{\delta a} \lesssim u\,\norm{a}.
\end{displaymath}
\end{theorem}

\begin{proof}
From Theorem~\ref{thm:backstabqr1} we know that $\tilde{p}$ is the characteristic polynomial  
of a matrix $A + \delta A = (Q + \delta Q)(R + \delta R)$, with $\norm{\delta{Q}} \lesssim u$.  
The form of $R + \delta{R}$ is given by Lemma~\ref{lem:rform2}.  
\begin{eqnarray*}
A + \delta A & = & (Q + \delta Q)(R + \delta R) \\
& = & (Q + \delta Q)[(I + \delta I) + (\alpha(x + \delta x)-e_{n})(y+\delta y)^{T}] \\
& = & (Q + \widetilde{\delta Q}) + (Q + \delta Q)(\alpha(x + \delta x) - e_{n})(y + \delta y)^{T},
\end{eqnarray*}
where $\widetilde{\delta Q} \doteq Q\,\delta I + \delta Q$ and $\norm{\widetilde{\delta Q}} \lesssim u$.
Now, using (\ref{eq:charpolform}) with $A$ replaced by $A+\delta A$, we get 
\begin{eqnarray}
\tilde{p}(\lambda)   & =  & \det(\lambda I - (A + \delta A)) \label{eq:pertpol}, \\
& = &  \sum_{k=0}^{n}\left[ (q_{k} + \delta q_{k}) + (y+\delta y)^{T}(G_{k} 
+ \delta G_{k})(Q + \delta Q)(\alpha\,(x + \delta x) - e_{n})\right]\lambda^{k}. \nonumber
\end{eqnarray}
where $\norm{\delta G_{k}} \lesssim u$ and $\absval{\delta q_{k}} \lesssim u$ by Lemma~\ref{lem:adjupert}. 
Expanding (\ref{eq:pertpol}) and ignoring higher order terms, we obtain 
\begin{eqnarray*}
\delta a_{k} & \doteq  & 
\delta q_{k} + \delta y^{T}\,G_{k}\,Q(\alpha\,x - e_{n}) 
+ y^{T}\delta G_{k}Q(\alpha\,x - e_{n}) \\
& & +\ y^{T}G_{k}\,\delta Q(\alpha\,x 
- e_{n}) + y^{T}G_{k}Q(\alpha\,\delta x)
\end{eqnarray*}
for $k=0$, \ldots, $n-1$.   Each term on the right-hand side has one factor that is $\lesssim u$.  All terms 
except the first contain 
exactly one factor $\alpha$ and other factors that are $\approx 1$.  Thus  
$\absval{\delta a_{k}} \lesssim u\,\absval{\alpha} = u\,\norm{a}$,  and therefore $\norm{\delta a} \lesssim u\, \norm{a}$.
\hfill\end{proof}

Figure~\ref{fig:backerrmonic} verifies that the backward error grows linearly 
in $\norm{a}$.  If we compare this with Figures~\ref{fig:1} and \ref{fig:2}, we see that the companion QR algorithm (with
no balancing step) has a significantly smaller backward error than 
the unstructured QR algorithm, with or without balancing.

\begin{figure}
  \centering
  \begin{tikzpicture}%
  \begin{loglogaxis}[%
    axis background/.style={fill=white},%
    clip mode=individual,
    name=plot55a,%
    anchor = west,%
    scale only axis,%
    xmin=1e0, xmax=1e12,%
    ymin=1e-18, ymax=1e+8,%
    yticklabel pos=left,
    ylabel={$\norm{a-\tilde{a}}$},%
    every axis y label/.style= {at={(-0.25,0.5)},rotate=90},
    width=0.36\textwidth,%
    height=0.18\textheight,%
    axis on top,%
    ]
    
    
    \addplot[SPECblue, semithick,only marks, mark=square]%
    table [x=normamon,y=QR||a-atilde||]{dat/be_0.dat};  
  
    \addplot[black, semithick] coordinates{
      (1e00,1e-14)
      (1e04,1e-10)
      (1e08,1e-06)
      (1e12,1e-02)
    };
    \node[coordinate,pin={[pin distance=.50cm]below:{$\norm{a}$}}]
    at (axis cs:  5e+8 , 5e-6 ){};
    \addplot[black, semithick] coordinates{
      (1e00,1e-14)
      (1e04,1e-06)
      (1e08,1e+02)
      (1e12,1e+10)
    };
    \node[coordinate,pin={[pin distance=.50cm]above:{$\norm{a}^{2}$}}]
    at (axis cs:  1e+7 , 1e-0 ){};


  \end{loglogaxis}
  \end{tikzpicture}
  \caption{Backward error of companion QR on the polynomial coefficients plotted against $\norm{a}$}\label{fig:backerrmonic}
\end{figure}
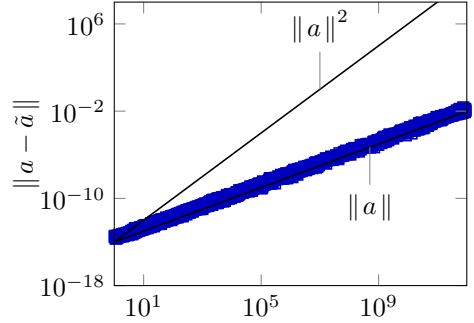

So far throughout this section 
we have assumed for convenience that we are dealing with a monic polynomial.  In practice we will 
often have a non-monic $p$, which we make monic by rescaling it.  The following theorem covers this case.

\begin{theorem}\label{thm:backerrnonmonic2}
Suppose we compute the roots of a non-monic polynomial $p$ with coefficient vector $a$ by applying
the companion QR algorithm to the monic polynomial $p/a_{n}$ with coefficient vector $a/a_{n}$.  
Let $\tilde{p}$ denote the monic polynomial that has the computed 
roots as its exact zeros, let $\hat{p} = a_{n}\tilde{p}$, and let $a + \delta a$ denote the coefficient
vector of $\hat{p}$.  Then 
\begin{displaymath}
\norm{\delta a} \lesssim u\,\norm{a}.
\end{displaymath}
\end{theorem}

\begin{proof}
Apply Theorem~\ref{thm:backerrmonic} to $p/a_{n}$, then rescale by multiplying by $a_{n}$.
\hfill\end{proof}

\subsection{Backward error of the companion QZ algorithm}

We now consider the backward error of the companion QZ algorithm, 
which finds the zeros of a non-monic polynomial
\begin{displaymath}
p(z) = a_{0} + a_{1}z + \cdots  + a_{n-1}z^{n-1} + a_{n}z^{n}
\end{displaymath}
by computing the eigenvalues of a pencil $V - \lambda W$ of the form
(\ref{eq:cpp}) with vectors $v$ and $w$ satisfying (\ref{eq:def:vw}).
We will assume a reasonable choice of $v$ and $w$ so that $\max\{\norm{v},\norm{w}\} 
\approx \norm{a}$, where $a$ is the coefficient vector of $p$ as before. 
(In fact, in all of our numerical experiments we have made the 
simplest choice, namely the one given by (\ref{eq:clp:top}).)  

Notice that in this setting we have the freedom to rescale the coefficients 
of the polynomial by an arbitrary factor.  Thus we can always arrange to
have $\norm{a} \approx 1$, for example.  This is the advantage of this approach,
and this is what allows us to get an optimal backward error bound in this case.

When we run the companion QZ algorithm on $(V,W)$, we obtain 
\begin{displaymath}
\hat{V} = U^{*}(V + \delta V)Z, \quad \hat{W} = U^{*}(W + \delta W)Z,
\end{displaymath}
where $\delta V$ and $\delta W$ are the backward errors.  We begin with 
an analogue of Theorem~\ref{thm:backstabqr1}.  

\begin{theorem}\label{thm:backstabqz1}
If the companion pencil eigenvalue problem
is solved by the companion QZ algorithm, then 
\begin{enumerate}
\item[(a)] $\hat{V} = U^{*}(V + \delta V)Z$, where $\norm{\delta V} \lesssim u \norm{a}$, 
\item[(b)] $\hat{W} = U^{*}(W + \delta W)Z$, where $\norm{\delta W} \lesssim u \norm{a}$. 
\end{enumerate}
\end{theorem}

\begin{proof}
The proof for $V=QR$ is identical to the proof of Theorem~\ref{thm:backstabqr1}.  The proof
for $W$ is even simpler because $W$ is already upper triangular;  there is no unitary $Q$ factor
to take into account.  
\end{proof}

The left panel of Figure~\ref{fig:backstabqz} gives numerical confirmation of
Theorem~\ref{thm:backstabqz1}.  We see that the growth is linear in $\norm{a}$
as claimed.

\begin{figure}
  \centering
  \begin{tikzpicture}%
    \begin{loglogaxis}[%
      name=numplot2,%
      clip mode=individual,
      scale only axis,%
      xmin=1e0, xmax=1e12,%
      ymin=1e-18, ymax=1e+8,%
      xticklabel pos=left,
      ylabel={Backward Error},%
      every axis y label/.style= {at={(-0.25,0.5)},rotate=90},%
      width=0.36\textwidth,%
      height=0.18\textheight,%
      axis on top,%
      legend columns=1,%
      legend style = {%
        fill opacity = 0.5,%
        draw opacity = 1,%
        text opacity = 1,%
        cells = {anchor = west},%
        legend pos = north west,%
        at = {(0.025,0.975)},%
      },%
      ]



      \addplot[SPECgreen, semithick,only marks, mark=pentagon]%
      table [x=norma,y=QZ||A-Ahat||]{dat/be_0.dat};  

      \addplot[SPECorange, semithick,only marks, mark=triangle]%
      table [x=norma,y=QZ||B-Bhat||]{dat/be_0.dat};

      \addplot[black, semithick] coordinates{
        (1e00,1e-14)
        (1e04,1e-10)
        (1e08,1e-06)
        (1e12,1e-02)
      };
      \node[coordinate,pin={[pin distance=.50cm]above:{$\norm{a}$}}]
      at (axis cs:  5e+9 , 5e-5 ){};
      \addplot[black, semithick] coordinates{
        (1e00,1e-14)
        (1e04,1e-06)
        (1e08,1e+02)
        (1e12,1e+10)
      };
      \node[coordinate,pin={[pin distance=.50cm]above:{$\norm{a}^{2}$}}]
      at (axis cs:  1e+7 , 1e-0 ){};

      \legend{$\norm{\delta V}_{F}$, $\norm{\delta W}_{F}$}
    \end{loglogaxis}

  \begin{loglogaxis}[%
    axis background/.style={fill=white},%
    clip mode=individual,
    name=plot510b,%
    at = {($(numplot2.east)+(+3pt,0)$)},
    anchor = west,%
    scale only axis,%
    xmin=1e0, xmax=1e12,%
    ymin=1e-18, ymax=1e+8,%
    yticklabel pos=right,
    ylabel={$\norm{a-\tilde{a}}$},
    every axis y label/.style= {at={(1.25,0.5)},rotate=-90},
    width=0.36\textwidth,%
    height=0.18\textheight,%
    axis on top,%
    ]
    
    
    \addplot[SPECblue, semithick,only marks, mark=square]%
    table [x=norma,y=QZ||a-ahat||]{dat/be_0.dat};  
    
    \addplot[black, semithick] coordinates{
      (1e00,1e-14)
      (1e04,1e-10)
      (1e08,1e-06)
      (1e12,1e-02)
    };
    \node[coordinate,pin={[pin distance=.50cm]below:{$\norm{a}$}}]
    at (axis cs:  5e+8 , 5e-6 ){};
    \addplot[black, semithick] coordinates{
      (1e00,1e-14)
      (1e04,1e-06)
      (1e08,1e+02)
      (1e12,1e+10)
    };
    \node[coordinate,pin={[pin distance=.50cm]above:{$\norm{a}^{2}$}}]
    at (axis cs:  1e+7 , 1e-0 ){};


  \end{loglogaxis}

  \end{tikzpicture}
  \caption{Backward errors $\norm{\delta V}_{F}$ and $\norm{\delta W}_{F}$ (both
    left) and $\norm{a -\hat{a}}$ (right) of companion QZ plotted against
    $\norm{a}$}
  \label{fig:backstabqz}
\end{figure}
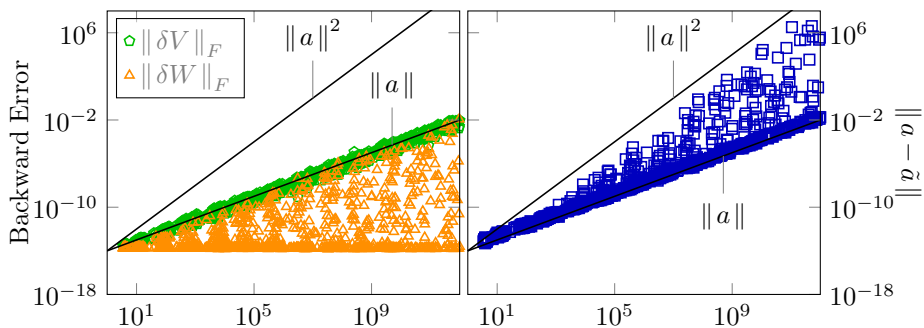

Now let us consider the backward error of the companion $QZ$ algorithm on the
polynomial coefficient vector $a$.  We will compute an optimally scaled backward
error as follows.  Given the computed roots, we build the monic polynomial
$\tilde{p}$ (with coefficient vector $\tilde{a}$) that has these as its exact
roots.  We then let $\hat{p} = \gamma \tilde{p}$ (with coefficient vector
$\hat{a}$), where $\gamma$ is chosen so that $\norm{a - \gamma\,\tilde{a}}$ is
minimized.  We hope to get a backward error $\norm{a - \hat{a}}$ that is linear
in $\norm{a}$, but the right panel of Figure~\ref{fig:backstabqz} shows what we actually get.  The
backward error seems to grow quadratically in $\norm{a}$, which is a
disappointment.

Combining Theorem~\ref{thm:backstabqz1} with the analysis of 
Edelman and Murakami \cite{m460}, we get the following
result.
\begin{theorem}\label{thm:backstabqz3}
  The backward error of the companion QZ algorithm on the polynomial coefficient
  vector satisfies
$$\norm{a - \hat{a}} \lesssim u\, \norm{a}^{2}.$$
\end{theorem}

So far it looks like companion QR is (surprisingly) more accurate than companion QZ, but we have not 
yet taken into account the freedom to rescale that we have in the companion QZ case.  

\begin{theorem}\label{thm:qzscale}
Suppose we compute the zeros of the polynomial $p$ with coefficient vector $a$ by applying the 
companion QZ algorithm to polynomial $p/\norm{a}$ with coefficient vector $a/\norm{a}$.  
Then the backward error satisfies 
$$\norm{a - \hat{a}} \lesssim u\, \norm{a}.$$
\end{theorem}

\begin{proof}
By Theorem~\ref{thm:backstabqz3} the backward error on $b=a/\|a\|$ satisfies 
$\norm{b - \hat{b}} \lesssim u\,\norm{b}^{2} = u$.  Therefore the backward error on $a$, which is 
$a - \hat{a} = \norm{a}(b - \hat{b})$ satisfies $\norm{a - \hat{a}} = \norm{a}\,\norm{b-\hat{b}}\lesssim u\,\norm{a}$.
\hfill\end{proof}

In the interest of full disclosure we must point out that this argument applies equally well to any stable method for computing
the eigenvalues of the pencil.   If we use, for example, the unstructured QZ algorithm on the rescaled 
polynomial $p/\norm{a}$, we will get the same result.  
Figure~\ref{fig:qzscale}  provides numerical confirmation of Theorem~\ref{thm:qzscale}.  In the left panel we
have the backward error for companion QZ, and in the right panel we have the backward error of the unstructured
QZ code from LAPACK.  
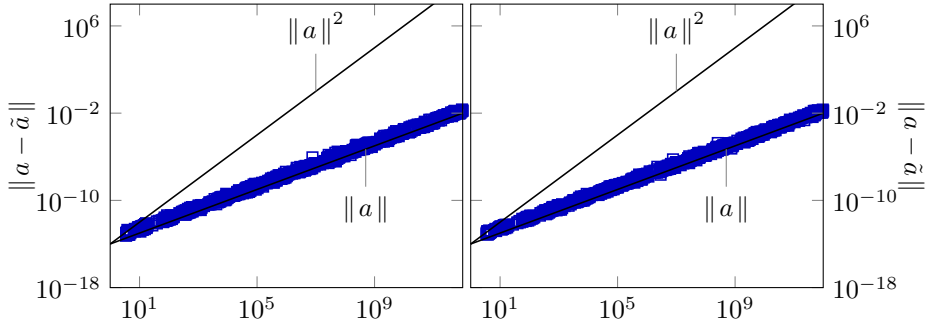
\begin{figure}
  \centering
  \begin{tikzpicture}%
  \begin{loglogaxis}[%
    axis background/.style={fill=white},%
    clip mode=individual,
    name=plot510a,%
    anchor = west,%
    scale only axis,%
    xmin=1e0, xmax=1e12,%
    ymin=1e-18, ymax=1e+8,%
    yticklabel pos=left,
    ylabel={$\norm{a-\tilde{a}}$},%
    every axis y label/.style= {at={(-0.25,0.5)},rotate=90},%
    width=0.36\textwidth,%
    height=0.18\textheight,%
    axis on top,%
    ]
    
    
    \addplot[SPECblue, semithick,only marks, mark=square]%
    table [x=norma,y=sQZ||a-ahat||]{dat/be_0.dat};  
        
    \addplot[black, semithick] coordinates{
      (1e00,1e-14)
      (1e04,1e-10)
      (1e08,1e-06)
      (1e12,1e-02)
    };
    \node[coordinate,pin={[pin distance=.50cm]below:{$\norm{a}$}}]
    at (axis cs:  5e+8 , 5e-6 ){};
    \addplot[black, semithick] coordinates{
      (1e00,1e-14)
      (1e04,1e-06)
      (1e08,1e+02)
      (1e12,1e+10)
    };
    \node[coordinate,pin={[pin distance=.50cm]above:{$\norm{a}^{2}$}}]
    at (axis cs:  1e+7 , 1e-0 ){};


  \end{loglogaxis}

  \begin{loglogaxis}[%
    axis background/.style={fill=white},%
    clip mode=individual,
    name=plot510b,%
    at = {($(plot510a.east)+(+3pt,0)$)},%
    anchor = west,%
    scale only axis,%
    xmin=1e0, xmax=1e12,%
    ymin=1e-18, ymax=1e+8,%
    yticklabel pos=right,
    ylabel={$\norm{a-\tilde{a}}$},%
    every axis y label/.style= {at={(1.25,0.5)},rotate=-90},%
    width=0.36\textwidth,%
    height=0.18\textheight,%
    axis on top,%
    ]
    
    
    \addplot[SPECblue, semithick,only marks, mark=square]%
    table [x=norma,y=sLQZ||a-ahat||]{dat/be_0.dat};  
    
    \addplot[black, semithick] coordinates{
      (1e00,1e-14)
      (1e04,1e-10)
      (1e08,1e-06)
      (1e12,1e-02)
    };
    \node[coordinate,pin={[pin distance=.50cm]below:{$\norm{a}$}}]
    at (axis cs:  5e+8 , 5e-6 ){};
    \addplot[black, semithick] coordinates{
      (1e00,1e-14)
      (1e04,1e-06)
      (1e08,1e+02)
      (1e12,1e+10)
    };
    \node[coordinate,pin={[pin distance=.50cm]above:{$\norm{a}^{2}$}}]
    at (axis cs:  1e+7 , 1e-0 ){};


  \end{loglogaxis}
  \end{tikzpicture}

  \caption{Backward error of scaled companion QZ algorithm (left) and unstructured QZ algorithm (right)}
  \label{fig:qzscale}
\end{figure}

\section{Conclusions}

The companion QR algorithm is not only faster than the unstructured QR algorithm, it also 
has a smaller backward error on the polynomial coefficients: 
$\norm{\tilde{a} - a} \lesssim u\,\norm{a}$.  In contrast the unstructured
QR algorithm only satisfies $\norm{\tilde{a}-a} \lesssim u\,\norm{a}^{2}$.  

As an alternative to the companion QR algorithm, we introduced a companion QZ algorithm that acts on a 
companion pencil.  Like the companion QR algorithm it runs in  $O(n^{2})$ time and uses $O(n)$ storage, but it 
is slower by a factor of about $5/3$.   The backward error of the companion QZ algorithm
also satisfies $\norm{\tilde{a} - a} \lesssim u\,\norm{a}$, provided that the polynomial is appropriately scaled 
before applying that algorithm.  

When we began this project we fully expected to find classes of problems for which the companion QZ algorithm
succeeds but companion QR fails.  So far we have not found any; the companion QR algorithm is much 
more robust than we had believed.  Since companion QR is faster, our recommendation at this time is to 
use companion QR and not companion QZ.  We do not exclude the possibility that classes of problems for
which companion QZ has superior performance will be found in the future. 

\section{Acknowledgments}

We thank the referees and the associate editor for comments that improved the paper.  In particular,
one referee pointed out that we can simplify the error analysis significantly by making use of the 
Faddeev-Leverrier  recurrence.  


\end{document}